\documentclass[12pt,reqno]{article}
\usepackage{amssymb,amscd,amsmath,amsthm,color}
\newtheorem{thm}{Theorem}[section]
\newtheorem{prop}[thm]{Proposition}
\newtheorem{lemma}[thm]{Lemma}

\newtheorem{remark}[thm]{Remark}
\newtheorem{example}[thm]{Example}
\newtheorem{problem}[thm]{Problem}
\numberwithin{equation}{section}

\def\bR{\mathbb{R}}

\def\bN{\mathbb{N}}
\def\cH{\mathcal{H}}

\def\bP{\mathbb{P}}

\def\<{\langle}
\def\>{\rangle}
\def\dT{d_{\mathrm T}}
\def\eps{\varepsilon}
\def\OM{\mathrm{OM}}

\def\bw{\mathbf{w}}
\def\cA{\mathcal{A}}

\def\cP{\mathcal{P}}
\def\cU{\mathcal{U}}

\begin{document}
\allowdisplaybreaks

\centerline{\LARGE Operator means deformed by a fixed point method}
\bigskip
\bigskip
\centerline{\Large
Fumio Hiai\footnote{{\it E-mail address:} hiai.fumio@gmail.com}}

\medskip
\begin{center}
$^1$\,Tohoku University (Emeritus), \\
Hakusan 3-8-16-303, Abiko 270-1154, Japan
\end{center}

\medskip
\begin{abstract}
By means of a fixed point method we discuss the deformation of operator means and multivariate
means of positive definite matrices/operators. It is shown that the deformation of an operator
mean becomes again an operator mean. The means deformed by the weighted power means are
particularly examined.

\bigskip\noindent
{\it 2010 Mathematics Subject Classification:}
Primary 47A64; Secondary 47B65, 47L07, 58B20

\bigskip\noindent
{\it Key words and phrases:}
Operator mean, operator monotone function, Positive definite matrices, Fixed point, Thompson
metric, Weighted geometric mean, Weighted power mean
\end{abstract}

\section{Introduction}

The notion of (two-variable) operator means of positive operators on a Hilbert space was
introduced in an axiomatic way by Kubo and Ando \cite{KA}. The main theorem of \cite{KA} says
that there is a one-to-one correspondence between the operator means $\sigma$ and the positive
operator monotone functions $f$ on $(0,1)$ with $f(1)=1$ in such a way that
\begin{align*}
A\sigma B=A^{1/2}f(A^{-1/2}BA^{1/2})A^{1/2}
\end{align*}
for positive invertible operators $A,B$ on a Hilbert space $\cH$. The extension to general
positive operators $A,B$ is given as $A\sigma B=\lim_{\eps\searrow0}(A+\eps I)\sigma(B+\eps I)$.
The operator monotone function $f$ on $(0,\infty)$ corresponding to $\sigma$ is called the
representing function of $\sigma$. Thus, most properties of operator means can be described in
terms of their representing functions, so the study of ``operator" means can essentially be
reduced to that of ``numerical" operator monotone functions on $(0,\infty)$. 

It was a long-standing problem to extend the notion of operator means to the case of more than
two variables of matrices/operators. A breakthrough came when the definitions of multivariate
geometric means of positive definite matrices were found in the iteration method by Ando, Li and
Mathias \cite{ALM} and in the Riemannian geometry method by Moakher \cite{Mo} and by Bhatia and
Holbrook \cite{BH}. In the latter method, the mean $G_\bw(A_1,\dots,A_n)$ with a probability
weight $\bw=(w_1,\dots,w_n)$ is defined as the minimizer of the weighted sum of the squares
$\sum_{j=1}^nw_j\delta^2(X,A_j)$, where $\delta(X,Y)$ is the Riemannian trace metric, and it is
also characterized by the gradient zero equation $\sum_{j=1}^nw_j\log X^{-1/2}A_jX^{-1/2}=0$
called the Karcher equation. Since then, the Riemannian multivariate means have extensively been
developed by many authors. Among others, the monotonicity property of $G_\bw(A_1,\dots,A_n)$
was proved in \cite{LL1}. In \cite{LP1} the multivariate weighted power means
$P_{\bw,r}(A_1,\dots,A_n)$ for $r\in[-1,1]\setminus\{0\}$ were introduced and the convergence
$\lim_{r\to0}P_{\bw,r}=G_\bw$ was proved. Furthermore, the multivariate geometric and power means
have recently been generalized to probability measures on the positive definite matrices based on
the Wasserstein distance (see, e.g., \cite{HL1,HL2,KL,LL4}).

In the recent study of Riemannian multivariate means there are two significant features worth
noting from the technical point of view. The one is that the positive definite matrices (of
fixed dimension) form a space of nonpositive curvature (NPC) so that the general theory of NPC
spaces \cite{St} is of essential use. The other is that the fixed point method is often used in
different places. For instance, the definition itself of the weighted power means $P_{\bw,r}$ is
done in terms of the fixed point (see Example \ref{E-multi-power} in Section 4). Motivated by
the second feature, we are aiming in the present paper to apply the fixed point method to
multivariate means as well as (two-variable) operator means in a more systematic way.

The paper is organized as follows. In Section 2, $\tau$ and $\sigma$ are operator means for
positive operators on $\cH$ assumed infinite-dimensional, where $\sigma\ne\frak{l}$ ($\frak{l}$
is the left trivial mean $X\frak{l}Y=X$). For any positive invertible operators $A,B$ we prove
that the fixed point equation $X=(X\sigma A)\tau(X\sigma B)$ has a unique positive invertible
solution, which is denoted by $A\tau_\sigma B$. Then $\tau_\sigma$ is proved to define an
operator mean. Section 3 presents properties and examples of the deformed operator means
$\tau_\sigma$. In particular, we examine the two-parameter deformation
$\tau_{s,r}:=\tau_{\frak{p}_{s,r}}$ for $s\in(0,1]$ and $r\in[-1,1]$ where $\frak{p}_{s,r}$ are
the weighted power means. In Section 4, we consider an $n$-variable mean $M$ of positive
invertible operators on $\cH$ (of finite or infinite dimension), and assume that $M$ satisfies
some basic properties such as joint monotonicity, etc. For any operator means
$\sigma_1,\dots,\sigma_n$ and positive invertible operators $A_1,\dots,A_n$, we show that the
equation $X=M(X\sigma_1A_1,\dots,X\sigma_nA_n)$ has a unique positive invertible solution,
which defines the deformed mean $M_{(\sigma_1,\dots,\sigma_n)}(A_1,\dots,A_n)$ satisfying
properties inherited from $M$. In particular, the deformation of $M$ by
$\sigma_1=\dots=\sigma_n=\frak{p}_{s,r}$ is examined. In this way, we can produce a lot of
multivariate means via deformation, while their construction is not so concrete. Since the
operator mean is a special case of multivariate means, there is a bit redundancy between the
presentations of Sections 2 and 3 and those of Section 4. However, the main stress in Sections 2
and 3 is the correspondence between the deformed operator means $\tau_\sigma$ and their
representing operator monotone functions, so the present way of separating sections makes the
paper more readable. Finally, a further extension of the fixed point method to the setting of
probability measures is briefly remarked in Section 5, whose detailed discussions will be in a
forthcoming paper.

\section{Main theorem}

Throughout this section we assume that $\cH$ is an infinite-dimensional separable Hilbert space
and $B(\cH)$ is the Banach space of all bounded linear operators on $\cH$ with the operator
norm $\|\cdot\|$. An operator $A\in B(\cH)$ is positive if $\<\xi,A\xi\>\ge0$ for all
$\xi\in\cH$. We denote by $B(\cH)^+$ the closed convex cone of positive operators in $B(\cH)$
and by $B(\cH)^{++}$ the open convex cone of invertible positive operators in $B(\cH)$.

The \emph{Thompson metric} on $B(\cH)^{++}$ is defined by
\begin{align*}
\dT(A,B):=\|\log A^{-1/2}BA^{-1/2}\|,\qquad A,B\in B(\cH)^{++},
\end{align*}
which is also written as
\begin{align*}
\dT(A,B)=\log\max\{M(A/B),M(B/A)\},
\end{align*}
where $M(A/B):=\inf\{\alpha>0:A\le\alpha B\}$. It is known \cite{Th} that $\dT$ is a complete
metric on $B(\cH)^{++}$ and the topology on $B(\cH)^{++}$ induced by $\dT$ is the same as the
operator norm topology.

An \emph{operator mean} $\sigma$ introduced by Kubo and Ando \cite{KA} is a binary operation
$$
\sigma:B(\cH)^+\times B(\cH)^+\ \longrightarrow\ B(\cH)^+
$$
satisfying the following four properties where $A,A',B,B',C\in B(\cH)^+$:
\begin{itemize}
\item[(I)] \emph{Joint monotonicity:} $A\le A'$ and $B\le B'$ imply $A\sigma B\le A'\sigma B'$.
\item[(II)] \emph{Transformer inequality:} $C(A\sigma B)C\le(CAC)\sigma(CBC)$.
\item[(III)] \emph{Downward continuity:} If $A_k,B_k\in B(\cH)^+$, $k\in\bN$, $A_k\searrow A$
and $B_k\searrow B$, then $A_k\sigma B_k\searrow A\sigma B$, where $A_k\searrow A$ means that
$A_1\ge A_2\ge\cdots$ and $A_k\to A$ in the strong operator topology.
\item[(IV)] \emph{Normalization:} $I\sigma I=I$, where $I$ is the identity operator on $\cH$.
\end{itemize}

Each operator mean $\sigma$ is associated with a non-negative operator monotone function
$f_\sigma$ on $[0,+\infty)$ with $f_\sigma(1)=1$, called the \emph{representing function} of
$\sigma$, in such a way that
\begin{align}\label{ope-mean}
A\sigma B=A^{1/2}f_\sigma(A^{-1/2}BA^{-1/2})A^{1/2},\qquad A,B\in B(\cH)^{++},
\end{align}
which is extended to general $A,B\in B(\cH)^+$ as
$A\sigma B:=\lim_{\eps\searrow0}(A+\eps I)\sigma(B+\eps I)$.

We write $\frak{l}$ and $\frak{r}$ to denote the two extreme left and right operator means,
i.e., $A\frak{l}B:=A$ and $A\frak{r}B:=B$ for every $A,B\in B(\cH)^+$. From now on, we assume
that $\tau$ and $\sigma$ are arbitrary operator means in the sense of Kubo and Ando, with a
single restriction $\sigma\ne\frak{l}$.

\begin{lemma}\label{L-Thomp}
Let $A,B,X,Y\in B(\cH)^{++}$. Then
\begin{itemize}
\item[(a)] $\dT(X\tau A,Y\tau B)\le\max\{\dT(X,Y),\dT(Y,B)\}$.
\item[(b)] $\dT(X\sigma A,Y\sigma A)<\dT(X,Y)$ whenever $X\ne Y$.
\end{itemize}
\end{lemma}

\begin{proof}
(a)\enspace
Set $\alpha:=\max\{\dT(X,Y),\dT(A,B)\}$. Since $e^{-\alpha}X\le Y\le e^\alpha X$ and
$e^{-\alpha}A\le B\le e^\alpha A$, we have
\begin{align*}
&Y\tau B\le(e^\alpha X)\tau(e^\alpha A)=e^\alpha(X\tau A), \\
&Y\tau B\ge(e^{-\alpha}X)\tau(e^{-\alpha}A)=e^{-\alpha}(X\tau A).
\end{align*}
Therefore, $\dT(X\tau A,Y\tau B)\le\alpha$.

(b)\enspace
In the following proof, we write $A<B$ to mean that $B-A\in B(\cH)^{++}$.
First, we see that if $X,A,B\in B(\cH)^{++}$ and $A<B$, then $X\sigma A<X\sigma B$. By
assumption $\sigma\ne\frak{l}$ or $f_\sigma\not\equiv1$, $f_\sigma$ is strictly increasing on
$[0,\infty)$. Since $A<B$ so that $X^{-1/2}AX^{-1/2}<X^{-1/2}BX^{-1/2}$, we have
$X^{-1/2}AX^{-1/2}+\eps I\le X^{-1/2}BX^{-1/2}$ for some $\eps>0$. Choose a $\delta\in(0,1)$
such that $\delta I\le X^{-1/2}AX^{-1/2}\le\delta^{-1}I$. Since
\begin{align*}
&f_\sigma(X^{-1/2}AX^{-1/2}+\eps I)-f_\sigma(X^{-1/2}AX^{-1/2}) \\
&\qquad\ge\bigl(\min\{f_\sigma(t+\eps)-f\sigma(t):\delta\le t\le\delta^{-1}\}\bigr)I,
\end{align*}
one has
\begin{align*}
f_\sigma(X^{-1/2}AX^{-1/2})<f_\sigma(X^{-1/2}AX^{-1/2}+\eps I)
\le f_\sigma(X^{-1/2}BX^{-1/2}),
\end{align*}
which implies that $X\sigma A<X\sigma B$.

Now, set $\alpha:=\dT(X,Y)>0$ (thanks to $X\ne Y$). Since $e^{-\alpha}A<A<e^\alpha A$, the
above shown fact gives
\begin{align*}
&Y\sigma A\le(e^\alpha X)\sigma A<(e^\alpha X)\sigma(e^\alpha A)=e^\alpha(X\sigma A), \\
&Y\sigma A\ge(e^{-\alpha}X)\sigma A>(e^{-\alpha}X)\sigma(e^{-\alpha}A)=e^{-\alpha}(X\sigma A).
\end{align*}
Therefore, $e^{-\beta}(X\sigma A)\le Y\sigma A\le e^\beta(X\sigma A)$ for some
$\beta\in(0,\alpha)$, which implies that $\dT(X\sigma A,Y\sigma A)\le\beta<\alpha$.
\end{proof}

\begin{remark}\label{R-conti}\rm
The inequality of Lemma \ref{L-Thomp}\,(a) implies that
$(A,B)\in(B(\cH)^{++})^2\to A\tau B\in B(\cH)^{++}$ is continuous in the operator norm, while it
is also obvious from expression \eqref{ope-mean}. It is worthwhile to note that when restricted
to $\delta I\le A,B\le\delta^{-1}I$ for any $\delta\in(0,1)$, the map $(A,B)\mapsto A\tau B$ is
continuous in the strong operator topology. To see this, recall a well-known fact that if $f$ is
a continuous function on a bounded interval $[a,b]$, then the functional calculus $A\mapsto f(A)$
on the self-adjoint operators $A$ with spectrum inside $[a,b]$ is continuous in the strong
operator topology. Hence, when restricted on $\delta I\le A,B\le\delta^{-1}I$, the map
$(A,B)\mapsto f_\tau(A^{-1/2}BA^{-1/2})$ is continuous in the strong operator topology, which
shows the assertion.
\end{remark}

In the rest of the section, we will consider, given $A,B\in B(\cH)^{++}$, the equation
\begin{align}\label{fixed-eq}
X=(X\sigma A)\tau(X\sigma B),\qquad X\in B(\cH)^{++}.
\end{align}

It might be instructive to start with a few typical examples of equation \eqref{fixed-eq} and
their solutions. Although the following examples are rather well-known, we discuss them in some
detail for the reader's convenience. Note in particular that the weighted power means were
introduced in \cite{LP1} by a fixed point method in the multivariable setting (see Example
\ref{E-multi-power} below).

\begin{example}\label{ex-1}\rm
Let $\tau=\#_\alpha$, the \emph{weighted geometric mean} with $0\le\alpha\le1$, i.e.,
\begin{align*}
A\#_\alpha B:=A^{1/2}(A^{-1/2}BA^{-1/2})^\alpha A^{1/2},\qquad A,B\in B(\cH)^{++},
\end{align*}
corresponding to the operator monotone function $t^\alpha$, $t\ge0$. Let $\sigma=\#_r$ with
$0<r\le1$. Then a unique solution to \eqref{fixed-eq} for any $A,B\in B(\cH)^{++}$ is
$X=A\#_\alpha B$. Indeed, in this case, \eqref{fixed-eq} is equivalent to
$I=(X^{-1/2}AX^{-1/2})^r\#_\alpha(X^{-1/2}BX^{-1/2})^r$. When $\alpha=0$, this means that
$X^{-1/2}AX^{-1/2}=I$ so that $X=A$. When $0<\alpha\le1$, the identity above means that
$X^{-1/2}BX^{-1/2}=(X^{-1/2}AX^{-1/2})^{1-{1\over\alpha}}$, so that
\begin{align*}
B=X^{1/2}(X^{-1/2}AX^{-1/2})^{1-{1\over\alpha}}X^{1/2}
=A^{1/2}(A^{-1/2}XA^{-1/2})^{1\over\alpha}A^{1/2},
\end{align*}
which is solved as $X=A\#_\alpha B$.
\end{example}

\begin{example}\label{ex-2}\rm
Let $\tau=\triangledown_\alpha$, the \emph{weighted arithmetic mean} where $0\le\alpha\le1$,
i.e., $A\triangledown_\alpha B:=(1-\alpha)A+\alpha B$, and $\sigma=\#_r$ with $0<r\le1$. With
$Y:=A^{-1/2}XA^{-1/2}$ and $C:=A^{-1/2}BA^{-1/2}$, \eqref{fixed-eq} in this case is equivalent
to
\begin{align*}
Y=(1-\alpha)Y\#_rI+\alpha Y\#_rC
\ &\iff\ I=(1-\alpha)Y^{-r}+\alpha(Y^{-1/2}CY^{-1/2})^r \\
\ &\iff\ C=\biggl[{Y^r-(1-\alpha)I\over\alpha}\biggr]^{1/r} \\
\ &\iff\ Y=[(1-\alpha)I+\alpha C^r]^{1/r}.
\end{align*}
Therefore, a unique solution to \eqref{fixed-eq} is $X=A\,\frak{p}_{\alpha,r}B$, where
$\frak{p}_{\alpha,r}$ is the \emph{weighted power mean}
\begin{align}\label{weighted-power}
A\,\frak{p}_{\alpha,r}B
:=A^{1/2}\bigl[(1-\alpha)I+\alpha(A^{-1/2}BA^{-1/2})^r\bigr]^{1/r}A^{1/2},
\quad A,B\in B(\cH)^{++},
\end{align}
corresponding to the operator monotone function $f_{\alpha,r}(t):=(1-\alpha+\alpha t^r)^{1/r}$,
$t\ge0$.
\end{example}

\begin{example}\label{ex-3}\rm
Let $\tau=\,!_\alpha$, the \emph{weighted harmonic mean} where $0\le\alpha\le1$, i.e.,
$A\,!_\alpha B:=[(1-\alpha)A^{-1}+\alpha B^{-1}]^{-1}$, and $\sigma=\#_r$ with $0<r\le1$. In this
case, \eqref{fixed-eq} is equivalent to
\begin{align*}
X^{-1}=(X^{-1}\#_rA^{-1})\triangledown_\alpha(X^{-1}\#_rB^{-1}),
\end{align*}
which has the unique solution $X=(A^{-1}\,\frak{p}_{\alpha,r}B^{-1})^{-1}$. The last expression
is indeed the weighted power mean
\begin{align}\label{weighted-power2}
A\,\frak{p}_{\alpha,-r}B
:=A^{1/2}\bigl[(1-\alpha)I+\alpha(A^{-1/2}BA^{-1/2})^{-r}\bigr]^{-1/r}A^{1/2},
\quad A,B\in B(\cH)^{++},
\end{align}
for $\alpha\in[0,1]$ and $-r\in[-1,0)$, corresponding to the operator monotone function
$f_{\alpha,-r}(t):=(1-\alpha+\alpha t^{-r})^{-1/r}$.
\end{example}

The main theorem (Theorem \ref{T-main}) of this section shows that equation \eqref{fixed-eq}
always has a unique solution and it indeed defines an operator mean. The proof of this will be
divided into several lemmas.

\begin{lemma}\label{L-fixed}
For every $A,B\in B(\cH)^{++}$ there exists a unique $X_0\in B(\cH)^{++}$ which satisfies
\eqref{fixed-eq}.
\end{lemma}

\begin{proof}
Define a map from $B(\cH)^{++}$ into itself by
\begin{align}\label{map-F}
F(X):=(X\sigma A)\tau(X\sigma B),\qquad X\in B(\cH)^{++},
\end{align}
which is monotone, i.e., $X\le Y$ implies $F(X)\le F(Y)$. Choose a $\delta\in(0,1)$ such that
$\delta I\le A,B\le\delta^{-1}I$, and let $Z:=\delta^{-1}I$. Since
$F(Z)=F(\delta^{-1}I)\le\delta^{-1}I=Z$, one has $Z\ge F(Z)\ge F(Z^2)\ge\cdots$. Moreover,
$F(Z)\ge F(\delta I)\ge\delta I$, and iterating this gives $F^k(Z)\ge\delta I$ for all $k$.
Therefore, $F^k(Z)\searrow X_0$ for some $X_0\in B(\cH)^{++}$. From the downward continuity of
(III) it follows that
\begin{align*}
F^{k+1}(Z)=(F^k(Z)\sigma A)\tau(F^k(Z)\sigma B)\ \searrow\ (X_0\sigma A)\tau(X_0\sigma B),
\end{align*}
which implies that $X_0=(X_0\sigma A)\tau(X_0\sigma B)$.

To prove the uniqueness, assume that $X_1\in B(\cH)^{++}$ satisfies \eqref{fixed-eq} and
$X_1\ne X_0$. By Lemma \ref{L-Thomp} we then have
\begin{align*}
\dT(X_0,X_1)\le\max\{\dT(X_0\sigma A,X_1\sigma A),\dT(X_0\sigma B,X_1\sigma B)\}<\dT(X_0,X_1),
\end{align*}
a contradiction.
\end{proof}

\begin{lemma}\label{L-fixed-ineq}
Let $A,B\in B(\cH)^{++}$ and $X_0\in B(\cH)^{++}$ be a solution to \eqref{fixed-eq} given in
Lemma \ref{L-fixed}.
\begin{itemize}
\item[(1)] If $Y\in B(\cH)^{++}$ and $Y\ge(Y\sigma A)\tau(Y\sigma B)$, then $Y\ge X_0$.
\item[(2)] If $Y'\in B(\cH)^{++}$ and $Y'\le(Y'\sigma A)\tau(Y'\sigma B)$, then $Y'\le X_0$.
\end{itemize}
\end{lemma}

\begin{proof}
Let $F$ be the map in \eqref{map-F}. To show (1), assume that $Y\in B(\cH)^{++}$ satisfies
$Y\ge F(Y)$ so that $Y\ge F(Y)\ge F^2(Y)\ge\cdots$. Choose a $\delta>0$ such that
$A,B,Y\ge\delta I$. It then follows as in the proof of Lemma \ref{L-fixed} that
$F^k(Y)\ge\delta I$ for all $k$, and therefore, $F^k(Y)\searrow Y_0$ for some $Y_0\in B(\cH)^{++}$.
Since $Y_0$ is a solution to \eqref{fixed-eq} as in the proof of Lemma \ref{L-fixed}, $Y_0=X_0$
so that $Y\ge Y_0=X_0$.

The proof of (2) is similar, where we have $Y'\le F(Y')\le F^2(Y')\le\cdots$ and use the upward
continuity of $\sigma,\tau$ that is a special case of continuity in the strong operator topology
noted in Remark \ref{R-conti}.
\end{proof}

For every $A,B\in B(\cH)^{++}$ we write $A\tau_\sigma B$ for the unique solution
$X_0\in B(\cH)^{++}$ to \eqref{fixed-eq} given in Lemma \ref{L-fixed}. Hence we have a binary
operation
\begin{align*}
\tau_\sigma:B(\cH)^{++}\times B(\cH)^{++}\ \longrightarrow\ B(\cH)^{++}.
\end{align*}

\begin{lemma}\label{L-properties}
The map $\tau_\sigma$ satisfies the following properties:
\begin{itemize}
\item[(i)]  Joint monotonicity: If $A\le A'$ and $B\le B'$ in $B(\cH)^{++}$, then
$A\tau_\sigma B\le A'\tau_\sigma B'$.
\item[(ii)] Transformer equality: For every invertible $A,B,C\in B(\cH)^{++}$,
\begin{align*}
C(A\tau_\sigma B)C=(CAC)\tau_\sigma(CBC).
\end{align*}
\item[(iii)] Downward continuity: If $A_k\searrow A$ and $B_k\searrow B$ in $B(\cH)^{++}$, then
$A_k\tau_\sigma B_k\searrow A\tau_\sigma B$.
\item[(iv)] Normalization: $I\tau_\sigma I=I$.
\end{itemize}
\end{lemma}

\begin{proof}
(i)\enspace
Let $X_0:=A\tau_\sigma B$ and $Y_0:=A'\tau_\sigma B'$. Since
\begin{align*}
Y_0=(Y_0\sigma A')\tau(Y_0\sigma B')\ge(Y_0\sigma A)\tau(Y_0\sigma B),
\end{align*}
we have $Y_0\ge X_0$ by Lemma \ref{L-fixed-ineq}\,(1).

(ii)\enspace
For $X_0:=A\tau_\sigma B$, from the transformer equality of $\tau,\sigma$ (for invertible $C$)
one has
\begin{align*}
CX_0C=\{(CX_0C)\sigma(CAC)\}\tau\{(CX_0C)\sigma(CBC)\}
\end{align*}
so that $CX_0C=(CAC)\tau_\sigma(CBC)$.

(iii)\enspace
One can choose a $\delta>0$ for which $A_k,B_k\ge\delta I$ for all $k$. Let
$X_k:=A_k\tau_\sigma B_k$. Since $X_k\ge\delta I$ for any $k$ (by the proof of Lemma
\ref{L-fixed}), it follows from joint monotonicity (i) above that $X_k\searrow X_0$ for some
$X_0\in B(\cH)^{++}$. Since
\begin{align*}
X_k=(X_k\sigma A_k)\tau(X_k\sigma B_k)\ \searrow\ (X_0\sigma A)\tau(X_0\sigma B),
\end{align*}
one has $X_0=(X_0\sigma A)\tau(X_0\sigma B)$ so that $X_0=A\tau_\sigma B$.

(iv) is clear.
\end{proof}

For every $A,B\in B(\cH)^+$, by monotonicity (i) of Lemma \ref{L-properties} we can define
$A\tau_\sigma B$ in $B(\cH)^+$ by
\begin{align}\label{ext-mean}
A\tau_\sigma B:=\lim_{\eps\searrow0}(A+\eps I)\tau_\sigma(B+\eps I).
\end{align} 
By property (iii) of Lemma \ref{L-properties} we see that this binary operation
$\tau_\sigma:B(\cH)^+\times B(\cH)^+\to B(\cH)^+$ extends the previous
$\tau_\sigma:B(\cH)^{++}\times B(\cH)^{++}\to B(\cH)^{++}$. Moreover, we have

\begin{lemma}\label{L-ext}
The map $\tau_\sigma$ extended to $B(\cH)^+\times B(\cH)^+$ satisfies joint monotonicity (I) and
downward continuity (III).
\end{lemma}

\begin{proof}
Joint monotonicity is obvious from (i) of Lemma \ref{L-properties} and definition
\eqref{ext-mean}. To prove downward continuity, assume that $A_k\searrow A$ and $B_k\searrow B$
in $B(\cH)^+$. For every $\xi\in\cH$ one has
\begin{align*}
\<\xi,(\lim_kA_k\tau_\sigma B_k)\xi\>
&=\inf_k\<\xi,(A_k\tau_\sigma B_k)\xi\> \\
&=\inf_k\inf_{\eps>0}\<\xi,\{(A_k+\eps I)\tau_\sigma(B_k+\eps I)\}\xi\> \\
&=\inf_{\eps>0}\inf_k\<\xi,\{(A_k+\eps I)\tau_\sigma(B_k+\eps I)\}\xi\> \\
&=\inf_{\eps>0}\<\xi,\{(A+\eps I)\tau_\sigma(B+\eps I)\}\xi\> \\
&=\<\xi,(A\tau_\sigma B)\xi\>,
\end{align*}
where the second and the last equalities are by definition \eqref{ext-mean} and the fourth
equality is due to (iii) of Lemma \ref{L-properties}. Therefore,
$A_k\tau_\sigma B_k\searrow A\tau_\sigma B$.
\end{proof}

A remaining requirement for $\tau_\sigma$ on $B(\cH)^+\times B(\cH)^+$ to be an operator mean is
transformer inequality (II). But it does not seem easy to directly prove inequality (II), so we
take a detour by giving the following lemma that was a key also in the proof of the main theorem
of \cite{KA}. 

\begin{lemma}\label{L-proj}
Let $A,B\in B(\cH)^{++}$ and $P\in B(\cH)$ be a projection commuting with $A,B$. Then $P$
commutes with $A\tau_\sigma B$ and
\begin{align}\label{fixed-mean-P}
\{(AP)\tau_\sigma(BP)\}=(A\tau_\sigma B)P.
\end{align}
\end{lemma}

\begin{proof}
Although we assumed that $\cH$ is an infinite-dimensional Hilbert space, the whole discussions
up to now are valid for any Hilbert space $\cH$. Let $\cH_0:=P\cH$ and $\cH_1:=(I-P)\cH$. One
can then apply Lemma \ref{L-fixed} to $AP,BP\in B(\cH_0)^{++}$ and
$A(I-P),B(I-P)\in B(\cH_1)^{++}$ to have an $X_0\in B(\cH_0)^{++}$ such that
\begin{align}\label{fixed-0}
X_0=\{X_0\sigma(AP)\}\tau\{X_0\sigma(BP)\},
\end{align}
and an $X_1\in B(\cH_1)^{++}$ such that
\begin{align*}
X_1=\{X_1\sigma(A(I-P))\}\tau\{X_1\sigma(B(I-P))\}.
\end{align*}
For any $\eps>0$, multiplying this with $\eps$ one further has
\begin{align}\label{fixed-2}
\eps X_1=\{(\eps X_1)\sigma(\eps A(I-P))\}\tau\{(\eps X_1)\sigma(\eps B(I-P))\}.
\end{align}
Since the operator means $\sigma,\tau$ are computed component-wise for direct sum operators
$Y_0+Y_1$ with $Y_i\in B(\cH_i)^+$, $i=0,1$, we add \eqref{fixed-0} and
\eqref{fixed-2} to have
$$
X_0+\eps X_1
=\{(X_0+\eps X_1)\sigma(AP+\eps A(I-P))\}\tau\{(X_0+\eps X_1)\sigma(BP+\eps B(I-P))\},
$$
which implies that
\begin{align}\label{mean-0-2}
X_0+\eps X_1=\{AP+\eps A(I-P)\}\tau_\sigma\{BP+\eps B(I-P)\},\qquad\eps>0.
\end{align}
Letting $\eps=1$ in \eqref{mean-0-2} gives
\begin{align}\label{mean-0-1}
X_0+X_1=A\tau_\sigma B.
\end{align}
From the downward continuity of $\tau_\sigma$ (Lemma \ref{L-ext}), letting $\eps\searrow0$ in
\eqref{mean-0-2} gives
\begin{align}\label{mean-0}
X_0=(AP)\tau_\sigma(BP).
\end{align}
From \eqref{mean-0-1} and \eqref{mean-0} we see that $P$ commutes with $A\tau_\sigma B$ and
\eqref{fixed-mean-P} holds.
\end{proof}

We are now in a position to present the main result of the section.

\begin{thm}\label{T-main}
Let $\tau$ and $\sigma$ be operator means with $\sigma\ne\frak{l}$. Then the binary operation
$\tau_\sigma$ first defined by equation \eqref{fixed-eq} for $A,B\in B(\cH)^{++}$ and then
extended by \eqref{ext-mean} for $A,B\in B(\cH)^+$ is an operator mean.

Moreover, the operator monotone function $f_{\tau_\sigma}$ corresponding to $\tau_\sigma$ is
determined in such a way that $x=f_{\tau_\sigma}(t)$ for $t>0$ is a unique solution to
\begin{align}\label{fixed-eq-scalar}
(x\sigma1)\tau(x\sigma t)=x,\qquad x>0,
\end{align}
that is,
\begin{align}\label{fixed-eq-scalar2}
f_\sigma(1/x)f_\tau\biggl({f_\sigma(t/x)\over f_\sigma(1/x)}\biggr)=1,\qquad x>0.
\end{align}
\end{thm}

\begin{proof}
Since Lemma \ref{L-proj} implies that $I\tau_\sigma(tI)$ for $t>0$ commutes with all
projections in $B(\cH)$, it follows that there is a function $f$ on $(0,\infty)$ such that
\begin{align*}
f(t)I=I\tau_\sigma(tI),\qquad t>0.
\end{align*}
Here, it is immediate to see that $x=f(t)$ is determined by the numerical equation
\begin{equation}\label{fixed-eq-scalar3}
x=(x\sigma1)\tau(x\sigma t),\qquad x>0,
\end{equation}
and it is also clear that $f(1)=1$. In the same way as in the proof of \cite[Theorem 3.6]{KA}
by using \eqref{fixed-mean-P} (where $P$ commutes $A,B$), one can show that, for every
$A\in B(\cH)^{++}$,
\begin{align}\label{ope-mono-f(A)}
f(A)=I\tau_\sigma A.
\end{align}
This implies that if $A,B\in B(\cH)^{++}$ and $A\le B$, then $f(A)\le f(B)$. Therefore, $f$ is
a positive operator monotone function on $(0,\infty)$, which can be extended to $[0,\infty)$ by
$f(0):=\lim_{t\searrow0}f(t)$.

Finally, by (ii) of Lemma \ref{L-properties} and \eqref{ope-mono-f(A)}, one can write for every
$A,B\in B(\cH)^{++}$
$$
A\tau_\sigma B=A^{1/2}(I\tau_\sigma(A^{-1/2}BA^{-1/2}))A^{1/2}
=A^{1/2}f(A^{-1/2}BA^{-1/2})A^{1/2}.
$$
Therefore, we have
$$
A\tau_\sigma B=Am_fB,\qquad A,B\in B(\cH)^{++},
$$
where $m_f$ is the operator mean corresponding to $f$. From the downward continuity of
$\tau_\sigma$ and $m_f$ on $B(\cH)^+\times B(\cH)^+$, the equality above extends to
$A,B\in B(\cH)^+$. Therefore, $\tau_\sigma=m_f$, that is, $\tau_\sigma$ is an operator mean
with $f_{\tau_\sigma}=f$. The determining equation \eqref{fixed-eq-scalar} was already shown in
\eqref{fixed-eq-scalar3}, and \eqref{fixed-eq-scalar2} is a more explicit rewriting of
\eqref{fixed-eq-scalar}.
\end{proof}

We call the operator mean $\tau_\sigma$ shown by Theorem \ref{T-main} the
\emph{deformed operator mean} from $\tau$ by $\sigma$.

\begin{remark}\label{R-generalization}\rm
More generally than \eqref{fixed-eq} one may consider the equation
\begin{align*}
X=(X\sigma_1A)\tau(X\sigma_2B),\qquad X\in B(\cH)^{++},
\end{align*}
where $\tau,\sigma_1,\sigma_2$ are operator means with $\sigma_1,\sigma_2\ne\frak{l}$. Then
the whole arguments of this section can similarly be done with this generalized equation, so
that one can define the operator mean $\tau_{(\sigma_1,\sigma_2)}$ deformed from $\tau$ by a
pair $(\sigma_1,\sigma_2)$. The corresponding operator monotone function
$f_{\tau_{(\sigma_1,\sigma_2)}}$ is determined in such a way that
$x=f_{\tau_{(\sigma_1,\sigma_2)}}(t)$ is a unique solution to $(x\sigma_1t)\tau(x\sigma_2t)=x$
or
\begin{align*}
f_{\sigma_1}(1/x)f_\tau\biggl({f_{\sigma_2}(t/x)\over f_{\sigma_1}(1/x)}\biggr)=1,
\qquad x>0.
\end{align*}
This generalized setting will be adopted in Section 4 to discuss deformation of multivariate
means of operators.

\end{remark}

\section{Properties and examples}

In this section we will show general properties of the deformed operator mean $\tau_\sigma$
and examine $\tau_\sigma$ when $\sigma$ varies over the weighted power means with two parameters.

There are three important transformations on the operator means \cite{KA}. For an operator mean
$\tau$, the {\it transpose} $\tau'$ of $\tau$ is defined as $A\tau'B:=B\tau A$, whose
representing operator monotone function is $f_{\tau'}(t):=tf_\tau(t^{-1})$. If
$\tau=\tau'$, $\tau$ is said to be {\it symmetric}. The {\it adjoint} $\tau^*$ of
$\tau$ is defined as $A\tau^*B:=(A^{-1}\tau B^{-1})^{-1}$, whose representing function
is $f_{\tau^*}(t)=f_\tau(t^{-1})^{-1}$. If $\tau=\tau^*$, $\tau$ is said to be
{\it self-adjoint}. The {\it dual} $\tau^\perp$ of $\tau$ is defined as
$\tau^\perp:=(\tau')^*=(\tau^*)'$, whose representing function is $t/f_\tau(t)$.

In this section, we assume as in Section 2 that $\tau,\sigma,\tau_1,\sigma_1$ are operator means
with $\sigma,\sigma_1\ne\frak{l}$ (recall that $\frak{l}$ and $\frak{r}$ are the left and the
right trivial means).

\begin{prop}\label{P-properties}
\begin{itemize}
\item[(1)] For $\sigma={\frak{r}}$, $\tau_{\frak{r}}=\tau$.
\item[(2)] $\frak{l}_\sigma=\frak{l}$ and $\frak{r}_\sigma=\frak{r}$.
\item[(3)] If $\tau\le\tau_1$ and $\sigma\le\sigma_1$, then $\tau_\sigma\le(\tau_1)_{\sigma_1}$.
\item[(4)] $(\tau_\sigma)'=(\tau')_\sigma$. Hence, if $\tau$ is symmetric, then so is
$\tau_\sigma$.
\item[(5)] $(\tau_\sigma)^*=(\tau^*)_{\sigma^*}$. Hence, if $\tau,\sigma$ are self-adjoint, then
so is $\tau_\sigma$.
\item[(6)] $(\tau_\sigma)^\perp=(\tau^\perp)_{\sigma^*}$. Hence, if $\tau=\tau^\perp$ and
$\sigma$ is self-adjoint, then $(\tau_\sigma)^\perp=\tau_\sigma$.
\end{itemize}
\end{prop}

\begin{proof}
(1) is trivial.

(2)\enspace
When $\tau=\frak{l}$, equation \eqref{fixed-eq} is $X=X\sigma A$, which has the solution $X=A$.
Similarly, when $\tau=\frak{r}$, \eqref{fixed-eq} has the solution $X=B$.

(3)\enspace
Let $A,B\in B(\cH)^{++}$ and set $X_0:=A\tau_\sigma B$ and $X_1:=A(\tau_1)_{\sigma_1}B$. Since
\begin{align*}
X_1=(X_1\sigma_1A)\tau_1(X_1\sigma_1B)\ge(X_1\sigma A)\tau(X_1\sigma B),
\end{align*}
one has $X_1\ge X_0$ by Lemma \ref{L-fixed-ineq}\,(1).

(4) is clear since
$$
(X\sigma A)\tau'(X\sigma B)=(X\sigma B)\tau(X\sigma A).
$$

(5) is clear since $X=(X\sigma^*A)\tau^*(X\sigma^*B)$ means that
$$
X^{-1}=(X^{-1}\sigma A^{-1})\tau(X^{-1}\sigma B^{-1}).
$$

(6) immediately follows from (4) and (5).
\end{proof}

\begin{prop}\label{deform-geom}
\begin{itemize}
\item[(1)] The representing function of $(\sigma^*)_\sigma$ is determined in such a way that
$x=f_{(\sigma^*)_\sigma}(t)$ is a solution of
$$
xf_\sigma(1/x)=f_\sigma(t/x),\quad\mbox{i.e.},\quad
f_{\sigma'}(x)=f_\sigma(t/x),\qquad x>0.
$$
Moreover, $(\sigma^*)_\sigma=\#$ if and only if $\sigma$ is symmetric.
\item[(2)] The representing function of $(\sigma^\perp)_\sigma$ is determined in such a way
that $x=f_{(\sigma^\perp)_\sigma}(t)$ is a solution of
$$
f_\sigma(1/x)=(x/t)f_\sigma(t/x),\quad\mbox{i.e.},\quad
f_\sigma(1/x)=f_{\sigma'}(x/t),\qquad x>0.
$$
Moreover, $(\sigma^\perp)_\sigma=\#$ if and only if $\sigma$ is symmetric.
\end{itemize}
\end{prop}

\begin{proof}
(1)\enspace
When $\tau=\sigma^*$, equation \eqref{fixed-eq-scalar2} means that
$$
f_\sigma(1/x)=f_\sigma\bigg({f_\sigma(1/x)\over f_\sigma(t/x)}\biggr),\qquad x>0,
$$
which is equivalent to
$$
1/x={f_\sigma(1/x)\over f_\sigma(t/x)},\quad\mbox{i.e.},\quad
xf_\sigma(1/x)=f_\sigma(t/x),\qquad x>0.
$$
Moreover, $(\sigma^*)_\sigma=\#$ holds if and only if the above holds for any $t=x^2$, that is
equivalent to $\sigma'=\sigma$.

(2)\enspace
When $\tau=\sigma^\perp$, \eqref{fixed-eq-scalar2} becomes
$$
{f_\sigma(t/x)\over f_\sigma\Bigl({f_\sigma(t/x)\over f_\sigma(1/x)}\Bigr)}=1,\qquad x>0,
$$
equivalently,
$$
t/x={f_\sigma(t/x)\over f_\sigma(1/x)},\quad\mbox{i.e.},\quad
f_\sigma(1/x)=(x/t)f_\sigma(t/x),\qquad x>0.
$$
Hence, $(\sigma^\perp)_\sigma=\#$ if and only if the above holds for any $t=x^2$, that is
equivalent to $\sigma'=\sigma$.
\end{proof}

The above proposition in particular says that
\begin{align}\label{harm-arith-geom}
\triangledown_{\,!}=\,!_\triangledown=\#_\#=\#.
\end{align}

In what follows we denote by $\OM_{+,1}(0,\infty)$ the set of non-negative operator monotone
functions $f$ on $[0,\infty)$ with $f(1)=1$, i.e., the set of representing operator monotone
functions of operator means. Recall \cite[(2.3.2)]{Hi} that if $f\in\OM_{+,1}(0,\infty)$ and
$f'(1)=\alpha$, then $\alpha\in[0,1]$ and
\begin{align}\label{alpha-min-max}
{t\over(1-\alpha)t+\alpha}\le f(t)\le(1-\alpha)+\alpha t,\qquad t\in[0,\infty).
\end{align}
When $f\in\OM_{+,1}(0,\infty)$, we note by \eqref{alpha-min-max} that
\begin{align}\label{min-max}
\min\{1,t\}\le f(t)\le\max\{1,t\},\qquad t\in[0,\infty),
\end{align}
and that if $f'(1)=0$ then $f\equiv1$.

The derivative $f_\tau'(1)$ ($\in[0,1]$) provides a significant characteristic of an operator
mean $\tau$. The next proposition says that this characteristic is preserved under taking the
deformed operator mean $\tau_\sigma$ for any $\sigma\ne\frak{l}$.

\begin{prop}\label{P-derivative}
$f_{\tau_\sigma}'(1)=f_\tau'(1)$ for every operator means $\tau,\sigma$ with $\sigma\ne\frak{l}$.
\end{prop}

\begin{proof}
First, assume that $f_{\tau_\sigma}'(1)=0$ and so $f_{\tau_\sigma}\equiv1$. Then $x=1$ is the
solution to \eqref{fixed-eq-scalar2} for every $t>0$, so that $f_\tau(f_\sigma(t))\equiv1$
for all $t>0$. Since $1$ is in the interior of the range of $f_\sigma$ thanks to
$\sigma\ne\frak{l}$, we have $f_\tau'(1)=0=f_{\tau_\sigma}'(1)$.

Next, assume that $f_{\tau_\sigma}'(1)>0$; then $f_{\tau_\sigma}$ is strictly increasing and
the inverse function $t=f_{\tau_\sigma}^{-1}(x)$ exists in a neighborhood of $1$. It follows from
\eqref{fixed-eq-scalar2} that
\begin{align*}
f_\sigma(1/x)f_\tau\biggl({f_\sigma(f_{\tau_\sigma}^{-1}(x)/x)\over f_\sigma(1/x)}\biggr)=1
\end{align*}
holds in a neighborhood of $1$. Now, noting $f_\tau(1)=f_\sigma(1)=f_{\tau_\sigma}^{-1}(1)=1$,
we differentiate the above at $x=1$ to obtain
\begin{align*}
-f_\sigma'(1)+f_\tau'(1)\bigl[f_\sigma'(1)\bigl(f_{\tau_\sigma}'(1)^{-1}-1\bigr)
+f_\sigma'(1)\bigr]=0,
\end{align*}
which, thanks to $f_\sigma'(1)>0$, implies that $f_{\tau_\sigma}'(1)=f_\tau'(1)$.
\end{proof}

\begin{prop}\label{inject}
For every $\sigma\ne\frak{l}$ the map $\tau\mapsto\tau_\sigma$ is injective on the operator
means.
\end{prop}

\begin{proof}
First, assume that $\tau_\sigma=\frak{l}$ or $f_{\tau_\sigma}\equiv1$. Then, as in the proof
of the previous proposition, $f_\tau(f_\sigma(t))\equiv1$ for all $t>0$. Since
$f_\sigma\not\equiv1$, we have $\tau=\frak{l}$.

Next, let $\kappa:=\tau_\sigma$ and assume $\kappa\ne\frak{l}$; then the range of $f_\kappa$
contains $[a,b]$ with $0<a<1<b$. As in the proof of the previous proposition, we have
\begin{align*}
f_\tau(\phi(x))={1\over f_\sigma(1/x)},\quad\mbox{where}
\quad\phi(x):={f_\sigma(f_\kappa^{-1}(x)/x)\over f_\sigma(1/x)},\quad x\in[a,b].
\end{align*}
Since $f_\kappa^{-1}(a)<1<f_\kappa^{-1}(b)$,
\begin{align*}
\phi(a)={f_\sigma(f_\kappa^{-1}(a)/a)\over f_\sigma(1/a)}<1,\qquad
\phi(b)={f_\sigma(f_\kappa^{-1}(b)/b)\over f_\sigma(1/b)}>1.
\end{align*}
Therefore, $f_\tau(t)$ on $[\phi(a),\phi(b)]$ is uniquely determined independently of $\tau$.
From the analyticity of $f_\tau$ this implies that $\tau$ is uniquely determined by $\kappa$
and $\sigma$.
\end{proof}

\begin{remark}\label{non-surject}\rm
The map $\tau\mapsto\tau_\sigma$ is not surjective onto the operator means in general. For
instance, when $\tau_\#=\kappa$, \eqref{fixed-eq-scalar2} becomes
$f_\tau(t^{1/2})=f_\kappa(t)^{1/2}$ or $f_\tau(t)=f_\kappa(t^2)^{1/2}$, $t>0$. If
$\kappa=\triangledown$, then $f_\tau(t)=\bigl({1+t^2\over2}\bigr)^{1/2}$, which is not operator
monotone. Hence there is no operator mean $\tau$ satisfying $\tau_\#=\triangledown$. Also,
assume that $\tau_\triangledown=\,!$; then \eqref{fixed-eq-scalar2} becomes
$f_\tau\bigl({x+t\over x+1}\bigr)={2x\over x+1}$ for $x={2t\over1+t}$. Therefore,
\begin{align*}
f_\tau\biggl({x(3-x)\over(x+1)(2-x)}\biggr)={2x\over x+1},\qquad0<x<2,
\end{align*}
from which we have $\lim_{t\searrow0}f_\tau(t)/t=4/3$. On the other hand, since $f_\tau'(1)=1/2$
by Proposition \ref{P-derivative}, it follows from \eqref{alpha-min-max} that
$\lim_{t\searrow0}f_\tau(t)/t\ge2$, a contradiction. Hence no operator mean $\tau$ satisfies
$\tau_\triangledown=\,!$. See Example \ref{E-three-lines} below for more about the deformed
operator means $\tau_\sigma$ by $\sigma=\#_r$ and $\sigma=\triangledown_r$ with $0<r\le1$.
\end{remark}

To show the continuous dependence of $\tau_\sigma$ on $\tau$ and $\sigma$, we prepare a basic
fact on convergence of operator means or their representing operator monotone functions. For
this we first give the next lemma.

\begin{lemma}\label{L-OM}
For every $f\in\OM_{+,1}(0,\infty)$,
\begin{align*}
f'(t)\le\max\{1,(2t)^{-1}\},\qquad t\in(0,\infty).
\end{align*}
Hence, for every $\delta\in(0,1)$,
\begin{align*}
\sup\{f'(t):t\ge\delta,\,f\in\OM_{+,1}(0,\infty)\}\le\delta^{-1}.
\end{align*}
\end{lemma}

\begin{proof}
It is well-known (see, e.g., \cite{Bh,Hi}) that an operator monotone function on $[0,\infty)$
admits an integral expression
\begin{align*}
f(t)=a+bt+\int_{(0,\infty)}{t(1+\lambda)\over t+\lambda}\,d\mu(\lambda),\qquad t\in[0,\infty),
\end{align*}
where $a\in\bR$, $b\ge0$ and $\mu$ is a positive finite measure on $(0,\infty)$. When
$f\in\OM_{+,1}(0,\infty)$, one has
\begin{align}\label{int-cond}
a=f(0)\ge0,\qquad f(1)=a+b+\mu((0,\infty))=1.
\end{align}
Compute
\begin{align*}
\phi_t(\lambda):={\partial\over\partial t}\,\biggl[{t(1+\lambda)\over t+\lambda}\biggr]
={\lambda(1+\lambda)\over(t+\lambda)^2}
\end{align*}
and
\begin{align*}
{d\over d\lambda}\,\phi_t(\lambda)={t+(2t-1)\lambda\over(t+\lambda)^3}.
\end{align*}
If $t\ge1/2$, then $\phi_t(\lambda)$ is monotone increasing in $\lambda\in(0,\infty)$ and hence
$\phi_t(\lambda)\le\lim_{\lambda\to\infty}\phi_t(\lambda)=1$. If $0<t<1/2$, then
\begin{align*}
\max_{\lambda\in(0,\infty)}\phi_t(\lambda)
=\phi_t\biggl({t\over1-2t}\biggr)={1\over4t(1-t)}\le{1\over2t}.
\end{align*}
These yield that $\phi_t(\lambda)\le\max\{1,(2t)^{-1}\}$ for all $t,\lambda\in(0,\infty)$.
Hence it follows from Lebesgue's dominated convergence theorem that
\begin{align*}
f'(t)&=b+\int_{(0,\infty)}\phi_t(\lambda)\,d\mu(\lambda) \\
&\le b+\max\{1,(2t)^{-1}\}\mu((0,\infty))\le\max\{1,(2t)^{-1}\},
\end{align*}
where the last inequality is due to \eqref{int-cond}.

The latter assertion is immediate since $\max\{1,(2t)^{-1}\}\le\delta^{-1}$ for any $t\ge\delta$
where $\delta\in(0,1)$.
\end{proof}

\begin{lemma}\label{L-OM-conv}
For operator means $\tau$ and $\tau_k$, $k\in\bN$, the following conditions are equivalent:
\begin{itemize}
\item[(a)] $f_{\tau_k}(t)\to f_\tau(t)$ for any $t\in(0,\infty)$;
\item[(b)] $f_{\tau_k}\to f_\tau$ uniformly on $[\delta,\delta^{-1}]$ for any $\delta\in(0,1)$;
\item[(c)] $A\tau_kB\to A\tau B$ in the norm for every $A,B\in B(\cH)^{++}$.
\end{itemize}
\end{lemma}

\begin{proof}
For any $\delta\in(0,1)$, since Lemma \ref{L-OM} shows that $f_{\tau_k}'(t)$ is uniformly
bounded for all $k$ and all $t\in[\delta,\delta^{-1}]$, it follows that $\{f_{\tau_k}\}$ is
equicontinuous on $[\delta,\delta^{-1}]$. Hence the pointwise convergence of (a) yields the
uniform convergence of $\{f_{\tau_k}\}$ on $[\delta,\delta^{-1}]$, so that we have
(a)$\iff$(b) since (b)$\implies$(a) is trivial.

Next, note that, for every $X\in B(\cH)^{++}$ with the spectrum $\sigma(X)$,
\begin{align*}
\|f_{\tau_k}(X)-f_\tau(X)\|=\sup_{t\in\sigma(X)}|f_{\tau_k}(t)-f_\tau(t)|,
\end{align*}
from which it is easy to see that (b)$\iff$(c).
\end{proof}

For $\tau,\tau_k$ as in Lemma \ref{L-OM-conv} we say that $\tau_k$ \emph{properly converges} to
$\tau$ and write $\tau_k\to\tau$ properly, if the equivalent conditions of the lemma hold. Note
that we do not take care of the convergence of $f_{\tau_k}(t)$ at $t=0$. For example, when
$\tau_k=(\#_{1/k}+\frak{r})/2$, we have $\tau_k\to\triangledown=(\frak{l}+\frak{r})/2$ properly
but $f_{\tau_k}(0)=0\not\to f_\triangledown(0)=1/2$.

\begin{prop}\label{P-conti}
Let $\tau,\sigma$ and $\tau_k,\sigma_k$, $k\in\bN$, be operator means with
$\sigma,\sigma_k\ne\frak{l}$, and assume that $\tau_k\to\tau$ and $\sigma_k\to\sigma$ properly.
Then $(\tau_k)_{\sigma_k}\to \tau_\sigma$ properly. Hence for every $A,B\in B(\cH)^{++}$,
$A(\tau_k)_{\sigma_k}B\to A\tau_\sigma B$ in the operator norm.
\end{prop}

\begin{proof}
For simplicity write $f_k:=f_{(\tau_k)_{\sigma_k}}$. By Lemma \ref{L-OM-conv} it suffices to
prove the pointwise convergence $f_k(t)\to f_{\tau_\sigma}(t)$ for any $t\in(0,\infty)$. For
each fixed $t\in(0,\infty)$ let $x_k:=f_k(t)$, $k\in\bN$. Then by \eqref{min-max} there is a
$\delta\in(0,1)$ such that $\delta\le x_k\le\delta^{-1}$, $k\in\bN$. Let $x_0$ be any limit
point of $\{x_k\}$ so that $x_0$ is a limit of a subsequence $\{x_{k_j}\}$. By
\eqref{fixed-eq-scalar2} for $\tau_{k_j},\sigma_{k_j}$ one has
\begin{align}\label{fixed-eq-k_j}
f_{\sigma_{k_j}}(1/x_{k_j})
f_{\tau_{k_j}}\Biggl({f_{\sigma_{k_j}}(t/x_{k_j})\over f_{\sigma_{k_j}}(1/x_{k_j})}\Biggr)
=1.
\end{align}
Since $f_{\sigma_{k_j}}\to f_\sigma$ properly, from (b) of Lemma \ref{L-OM-conv} one can easily
see that
\begin{align*}
f_{\sigma_{k_j}}(x_{k_j})\ \longrightarrow\ f_\sigma(x_0),\qquad
f_{\sigma_{k_j}}(t/x_{k_j})\ \longrightarrow\ f_\sigma(t/x_0).
\end{align*}
Moreover, since $f_{\tau_{k_j}}\to f_\tau$ properly, one can similarly have
\begin{align*}
f_{\tau_{k_j}}\Biggl({f_{\sigma_{k_j}}(t/x_{k_j})\over f_{\sigma_{k_j}}(1/x_{k_j})}\Biggr)
\ \longrightarrow\ f_\tau\biggl({f_\sigma(t/x_0)\over f_\sigma(x_0)}\biggr).
\end{align*}
Hence, letting $j\to\infty$ in \eqref{fixed-eq-k_j} gives
\begin{align*}
f_\sigma(1/x_0)f_\tau\biggl({f_\sigma(t/x_0)\over f_\sigma(1/x_0)}\biggr)=1
\end{align*}
so that $x_0=f_{\tau_\sigma}(t)$. Therefore, we find that $f_{\tau_\sigma}(t)$ is a unique limit
point of $\{x_k\}$ so that $x_k=f_k(t)\to f_{\tau_\sigma}(t)$, as desired.
\end{proof}

In the rest of this section we will discuss the deformed operator means from an arbitrary $\tau$
by the weighted power means $\frak{p}_{s,r}$ with two parameters $s\in(0,1]$ and $r\in[-1,1]$.
Recall that $\frak{p}_{s,r}$ is the operator mean corresponding to the operator monotone
function
\begin{align*}
f_{s,r}(t):=(1-s+st^r)^{1/r},\qquad t\in[0,\infty),
\end{align*}
see \eqref{weighted-power} and \eqref{weighted-power2}. Here, we use the convention that
\begin{align}\label{convention}
\frak{p}_{s,0}:=\#_s,
\end{align}
that is justified as $\lim_{r\to0}f_{s,r}(t)=t^s$ for any $t\in[0,\infty)$ so that
$\frak{p}_{s,r}\to\#_s$ properly as $r\to0$. Also, we restrict $s$ to $(0,1]$ in view of
$\frak{p}_{0,r}=\frak{l}$.

For each operator mean $\tau$, we introduce the two-parameter deformation of $\tau$ as
\begin{align}\label{tau-s,r}
\tau_{s,r}:=\tau_{\frak{p}_{s,r}},\qquad s\in(0,1],\ r\in[-1,1].
\end{align}
In particular, we have
\begin{align*}
\tau_{s,-1}=\tau_{\,!_s},\qquad
\tau_{s,0}=\tau_{\#_s},\qquad
\tau_{s,1}=\tau_{\triangledown_s},
\end{align*}
the deformed operator means by the weighted harmonic, the weighted geometric and the weighted
arithmetic means, respectively. Moreover,
\begin{align*}
\tau_{1,r}=\tau_{\frak{r}}=\tau,\qquad r\in[-1,1].
\end{align*}
Note that if $s_k\in(0,1]$ and $r_k\in[-1,1]$, $k\in\bN$, are such that $s_k\to s\in(0,1]$
and $r_k\to r\in[-1,1]$, then $f_{s_k,r_k}(t)\to f_{s,r}(t)$ for any $t\in(0,\infty)$ and
hence Proposition \ref{P-conti} implies that $\tau_{s_k,r_k}\to\tau_{s,r}$ properly. Thus, the
two-parameter deformation $\tau_{s,r}$ ($0<s\le1$, $-1\le r\le1$) of $\tau$ is a continuous
family of operator means bounded with $\tau$ (at $s=1$).

An interesting question here is to find which operator mean appears as the boundary value of
$\tau_{s,r}$ in the limit $s\to0$, which we settle in the following theorem.

\begin{thm}\label{T-tau-boundary}
Let $s_k\in(0,1]$ and $r_k\in[-1,1]$, $k\in\bN$, be such that $s_k\to0$ and $r_k\to r$. Then
$\tau_{s_k,r_k}\to\frak{p}_{\alpha,r}$ properly, where $\alpha:=f_\tau'(1)\in[0,1]$. Thus,
letting
\begin{align}\label{tau-s,r-boundary}
\tau_{0,r}=\frak{p}_{\alpha,r},\qquad r\in[-1,1],
\end{align}
we have a continuous (in the sense of Lemma \ref{L-OM-conv}) family of operator means
$\tau_{s,r}$ with two parameters $s\in[0,1]$ and $r\in[-1,1]$.
\end{thm}

\begin{proof}
For simplicity write $f_k:=f_{s_k,r_k}$. By Lemma \ref{L-OM-conv} it suffices to show the
convergence $f_k(t)\to f_{\alpha,r}(t)$ for any fixed $t\in(0,\infty)$. Let $x_k:=f_k(t)$,
$k\in\bN$; then as in the proof of Proposition \ref{P-conti}, there is a $\delta\in(0,1)$ such
that $\delta\le x_k\le\delta^{-1}$, $k\in\bN$. It remains to prove that $f_{\alpha,r}(t)$ is a
unique limit point of $\{x_k\}$. For this, by replacing $\{x_k\}$ with a subsequence, we may
assume (for notational brevity) that $\{x_k\}$ itself converges to some $x_0$. By
\eqref{fixed-eq-scalar2} one has
\begin{align}\label{fixed-eq-s,r}
f_k(1/x_k)f_\tau\biggl({f_k(t/x_k)\over f_k(1/x_k)}\biggr)=1.
\end{align}

First, assume that $r\ne0$. Since $r_k\to r$, we may assume that $r_k\in[r/2,1]$, $k\in\bN$,
for $r>0$ and $r_k\in[-1,r/2]$, $k\in\bN$, for $r<0$. Since $s_k\to0$, one has
\begin{align*}
f_k(1/x_k)=\biggl(1-s_k+{s_k\over x_k^{r_k}}\biggr)^{1/r_k}
=1+{s_k\over r_k}\biggl({1\over x_k^{r_k}}-1\biggr)+o(s_k)
\quad\mbox{as $k\to\infty$},
\end{align*}
where $o(s_k)/s_k\to0$ as $k\to\infty$. Similarly,
\begin{align*}
f_k(t/x_k)=1+{s_k\over r_k}\biggl({t^{r_k}\over x_k^{r_k}}-1\biggr)+o(s_k)
\quad\mbox{as $k\to\infty$},
\end{align*}
so that
\begin{align*}
{f_k(t/x_k)\over f_k(1/x_k)}
=1+{s_k\over r_k}\biggl({t^{r_k}\over x_k^{r_k}}-{1\over r_k^{r_k}}\biggr)+o(s_k)
\quad\mbox{as $k\to\infty$}.
\end{align*}
Since $f_\tau(1+x)=1+\alpha x+o(x)$ as $x\to0$, one furthermore has
\begin{align*}
f_\tau\biggl({f_k(t/x_k)\over f_k(1/x_k)}\biggr)
=1+{\alpha s_k\over r_k}\biggl({t^{r_k}\over x_k^{r_k}}-{1\over x_k^{r_k}}\biggr)+o(s_k)
\quad\mbox{as $k\to\infty$}.
\end{align*}
Therefore, by \eqref{fixed-eq-s,r} we obtain
\begin{align*}
\biggl[1+{s_k\over r_k}\biggl({1\over x_k^{r_k}}-1\biggr)+o(s_k)\biggr]
\biggl[1+{\alpha s_k\over r_k}\biggl({t^{r_k}\over x_k^{r_k}}
-{1\over x_k^{r_k}}\biggr)+o(s_k)\biggr]=1,
\end{align*}
that is,
\begin{align*}
{s_k\over r_k}\biggl({1\over x_k^{r_k}}-1\biggr)
+{\alpha s_k\over r_k}\biggl({t^{r_k}\over x_k^{r_k}}-{1\over x_k^{r_k}}\biggr)=o(s_k)
\quad\mbox{as $k\to\infty$}.
\end{align*}
This implies that
\begin{align*}
\lim_{k\to\infty}\biggl[{1\over r_k}\biggl({1\over x_k^{r_k}}-1\biggr)
+{\alpha\over r_k}\biggl({t^{r_k}\over x_k^{r_k}}-{1\over x_k^{r_k}}\biggr)\biggr]=0.
\end{align*}
Therefore,
\begin{align*}
{1\over r}\biggl({1\over x_0^r}-1\biggr)
+{\alpha\over r}\biggl({t^r\over x_0^r}-{1\over x_0^r}\biggr)=0,
\end{align*}
that is, $1-x_0^r+\alpha(t^r-1)=0$, implying that
$x_0=(1-\alpha+\alpha t^r)^{1/r}=f_{\alpha,r}(t)$,

Next, assume that $r=0$, so $s_k,r_k\to0$. Since, as $k\to\infty$,
\begin{align*}
{1\over x_k^{r_k}}-1=e^{-r_k\log x_k}-1=-r_k\log x_k+o(r_k),
\end{align*}
one can compute
\begin{align*}
\log f_k(1/x_k)
&={1\over r_k}\log\biggl[1+s_k\biggl({1\over x_k^{r_k}}-1\biggr)\biggr] \\
&={1\over r_k}\log\bigl[1-s_kr_k\log x_k+s_ko(r_k)\bigr] \\
&={1\over r_k}\bigl[-s_kr_k\log x_k+s_ko(r_k)+o(s_kr_k)\bigr] \\
&=-s_k\log x_k+o(s_k),
\end{align*}
where the last equality follows from
\begin{align*}
{1\over s_k}\biggl[{s_ko(r_k)+o(s_kr_k)\over r_k}\biggr]
={o(r_k)\over r_k}+{o(s_kr_k)\over s_kr_k}\ \longrightarrow\ 0\quad\mbox{as $k\to\infty$}.
\end{align*}
Therefore,
\begin{align*}
f_k(1/x_k)=\exp\bigl[-s_k\log x_k+o(s_k)\bigr]
=1-s_k\log x_k+o(s_k)\quad\mbox{as $k\to\infty$},
\end{align*}
and similarly,
\begin{align*}
f_k(t/x_k)=1-s_k\log{x_k\over t}+o(s_k)\quad\mbox{as $k\to\infty$}.
\end{align*}
Hence we obtain, as in the proof in the case $r\ne0$,
\begin{align*}
\bigl[1-s_k\log x_k+o(s_k)\bigr]
\biggl[1+\alpha s_k\biggl(\log x_k-\log{x_k\over t}\biggr)+o(s_k)\biggr]=1,
\end{align*}
that is, $-s_k\log x_k+\alpha s_k\log t=o(s_k)$ as $k\to\infty$. Therefore,
$\lim_{k\to\infty}(-\log x_k+\alpha\log t)=0$, implying that $x_0=t^\alpha=f_{\alpha,0}(t)$.

The latter statement of the theorem is an immediate consequence of the convergence property
just proved as well as the fact remarked just before the theorem.
\end{proof}

The deformed operator means $\tau_{s,r}$ ($0\le s\le1$, $-1\le r\le1$) constructed above are
drawn in the following figure:
$$
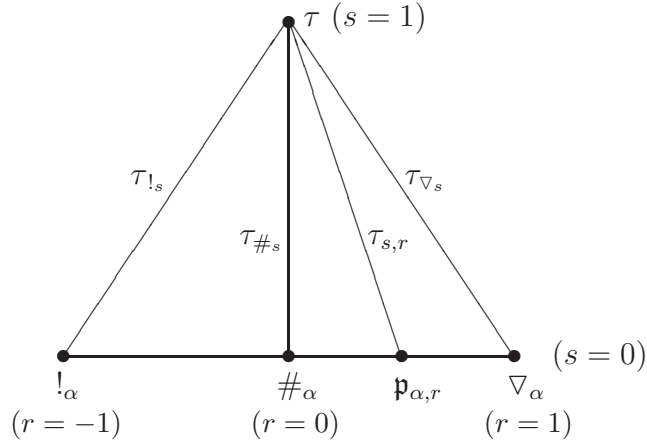
\setlength{\unitlength}{1mm}
\begin{picture}(50,60)(5,-13)
\put(0,0){\line(1,0){60}}
\put(30,0){\line(0,1){45}}
\put(0,0){\line(2,3){30}}
\put(60,0){\line(-2,3){30}}
\put(45,0){\line(-1,3){15}}
\put(-1,-1){$\bullet$}
\put(29,-1){$\bullet$}
\put(59,-1){$\bullet$}
\put(44,-1){$\bullet$}
\put(29,43.5){$\bullet$}
\put(32,44){$\tau$\ ($s=1$)}
\put(65,-1){($s=0$)}
\put(59,-5){$\triangledown_\alpha$}
\put(56,-10){\small($r=1$)}
\put(28.5,-5){$\#_\alpha$}
\put(25,-10){\small($r=0$)}
\put(-1,-5){$!_\alpha$}
\put(-7,-10){\small($r=-1$)}
\put(44,-5){$\frak{p}_{\alpha,r}$}
\put(9,23){$\tau_{\,!_s}$}
\put(23.5,15){$\tau_{\#_s}$}
\put(40.5,15){$\tau_{s,r}$}
\put(45.5,23){$\tau_{\triangledown_s}$}
\end{picture}
$$

\begin{example}\label{E-three-lines}\rm
We here examine the representing operator monotone functions of the deformed operator means
in the three lines $\tau_{s,1}=\tau_{\triangledown_s}$, $\tau_{s,0}=\tau_{\#_s}$ and
$\tau_{s,-1}=\tau_{\,!_s}$ with $s\in[0,1]$. We write $f=f_\tau$ that can be an arbitrary
element of $\OM_{+,1}(0,\infty)$ with $\alpha=f'(1)$.

(1)\enspace
When $\sigma=\#_r$ with $0<r\le1$, equation \eqref{fixed-eq-scalar2} is solved as
$x=f(t^r)^{1/r}$. It is well-known that $f(t^r)^{1/r}\in\OM_{+,1}(0,\infty)$ for any
$r\in(0,1]$, but it seems less well-known that $f(t^r)^{1/r}\to t^\alpha$ properly as
$r\searrow0$, a particular case of Theorem \ref{T-tau-boundary}. One can define a
one-parameter continuous family of ``generalized power means" $\frak{p}_{\tau,r}$ for
$r\in[-1,1]$ by
\begin{align*}
\frak{p}_{\tau,r}:=\begin{cases}\tau_{\#_r} & (0<r\le1), \\
\#_\alpha  & (r=0), \\
(\tau^*)_{\#_{-r}} & (-1\le r<0),\end{cases}
\end{align*}
joining $\tau$ ($r=1$), $\#_\alpha$ ($r=0$) and $\tau^*$ ($r=-1$), whose representing function
is $f(r^r)^{1/r}$ for $r\ne0$. A special case where $\tau=\triangledown_\alpha$ is the
weighted power means $\frak{p}_{\alpha,r}$ for $r\in[-1,1]$ with $\alpha$ fixed, dealt with in
Examples \ref{ex-2} and \ref{ex-3}.

(2)\enspace
When $\sigma=\triangledown_s$ with $0<s\le1$, equation \eqref{fixed-eq-scalar2} becomes
\begin{align}\label{fixed-eq-particular}
f\biggl({(1-r)x+rt\over(1-s)x+s}\biggr)={x\over(1-s)x+s},\qquad x>0,
\end{align}
whose solution is $x=f_{\tau_{\triangledown_s}}(t)$. For instance, when $\tau=\,!_\alpha$,
the above equation means that
\begin{align*}
{(1-s)x+s(1-\alpha)t+s\alpha\over(1-s)x+st}={(1-s)x+s\over x}.
\end{align*}
Solving this we find that the representing function of $(!_\alpha)_{\triangledown_s}$ (where
$0<s<1$) is
\begin{align}\label{harm-arith}
f_{(!_\alpha)_{\triangledown_s}}(t)
={1-\alpha-s+(\alpha-s)t+\sqrt{(1-\alpha-s+(\alpha-s)t)^2+4s(1-s)t}\over2(1-s)}.
\end{align}
When $s=0$, the above right-hand side reduces to $f_{\triangledown_\alpha}(t)$, which is
compatible with the fact that $(!_\alpha)_{\triangledown_s}$ approaches to
$\triangledown_\alpha$ as $s\searrow0$, a particular case of Theorem \ref{T-tau-boundary}.
In particular, when $\sigma=\triangledown_s$ and $\tau=\#$, \eqref{fixed-eq-particular} means
that
\begin{align*}
[(1-s)x+s][(1-s)x+st]=x^2.
\end{align*}
Solving this shows that the representing function of $\#_{\triangledown_s}$ is
\begin{align}\label{geom-arith}
f_{\#_{\triangledown_s}}(t)
={(1-s)(1+t)+\sqrt{(1-s)^2(1+t)^2+4s(2-s)t}\over2(2-s)}.
\end{align}
Note that when $s=0$ the above reduces to $f_\triangledown(t)=(1+t)/2$, which is compatible
with the fact that $\#_{\triangledown_s}$ approaches to $\triangledown$ as $s\searrow0$. 

(3)\enspace
When $\sigma=\,!_s$ with $0<s\le1$, \eqref{fixed-eq-scalar2} becomes
\begin{align*}
f\biggl({(1-s)t+stx\over(1-s)t+sx}\biggr)=1-s+sx,
\end{align*}
whose solution is $x=f_{\tau_{\,!_s}}(t)$. Note that $x=f_{\tau_{\,!_s}}(t)$ is equivalent to
$x^{-1}=f_{(\tau^*)_{\triangledown_s}}(t^{-1})$, as seen from Proposition
\ref{P-properties}\,(5). For instance, we have
$f_{\#_{\,!_s}}(t)=f_{\#_{\triangledown_s}}(t^{-1})^{-1}$ and
$f_{(\triangledown_\alpha)_{\,!_s}}(t)=f_{(!_\alpha)_{\triangledown_s}}(t^{-1})^{-1}$, whose
explicit forms can be computed from \eqref{geom-arith} and \eqref{harm-arith}.
\end{example}

We end the section with some discussions on one-parameter families of operator means. Let
$\{m_\alpha\}_{\alpha\in[0,1]}$ be a one-parameter continuous (in the sense of Lemma
\ref{L-OM-conv}) family of operator means. We say that such a family is \emph{regular} if
$f_{m_\alpha}'(1)=\alpha$ (equivalently, $!_\alpha\le m_\alpha\le\triangledown_\alpha$, as
noted in \eqref{alpha-min-max}) for all $\alpha\in[0,1]$. Note that the condition in particular
contains $m_0=\frak{l}$ and $m_1=\frak{r}$. In \cite{FK} Fujii and Kamei constructed, given a
symmetric operator mean $\sigma$, a regular continuous family of operator means
$\{m_\alpha\}_{\alpha\in[0,1]}$ in the following way:
$$
m_0:=\frak{l},\qquad m_{1/2}:=\sigma,\qquad m_1:=\frak{r},
$$
and inductively
$$
Am_{2k+1\over2^{n+1}}B:=(Am_{k\over2^n}B)m(Am_{k+1\over2^n}B)
$$
for $n,k\in\bN$ with $2k+1<2^{n+1}$. The construction was extended by P\'alfia and Petz \cite{PP}
to an arbitrary (not necessarily symmetric) operator mean $\sigma$ ($\ne\frak{l},\frak{r}$) in
such a way that $m_s=\sigma$ when $s=f_\sigma'(1)$; see \cite{UYY} for the equivalence between
the two constructions for a symmetric $\sigma$. By Propositions \ref{P-derivative} and
\ref{P-conti} note also that if $\{m_\alpha\}_{\alpha\in[0,1]}$ is such a regular continuous
family, then so is the deformed $\{(m_\alpha)_\sigma\}_{\alpha\in[0,1]}$ by any
$\sigma\ne\frak{l}$; for instance, $\{(!_\alpha)_{\triangledown_s}\}_{\alpha\in[0,1]}$ for any
$s\in(0,1]$ given in \eqref{harm-arith}.

Here, of our particular concern is a family of operator means having the
interpolation property, introduced in \cite{Fu,FK} as follows: A continuous family of operator
means $\{m_\alpha\}_{\alpha\in[0,1]}$ is called an \emph{interpolation} family if
\begin{align*}
(am_\alpha b)m_\delta(am_\beta b)=am_{(1-\delta)\alpha+\delta\beta}b,
\qquad a,b\in(0,\infty),
\end{align*}
for all $\alpha,\beta,\delta\in[0,1]$. In this case, for every $\alpha,\beta\in[0,1]$ one has
\begin{align*}
\{(am_\alpha b)m_\beta a\}m_\alpha\{(am_\alpha b)m_\beta b\}
&=\{(am_\alpha b)m_\beta(am_0b)\}m_\alpha\{(am_\alpha b)m_\beta(am_1b)\} \\
&=(am_{(1-\beta)\alpha}b)m_\alpha(am_{(1-\beta)\alpha+\beta}b) \\
&=am_{(1-\alpha)(1-\beta)\alpha+\alpha((1-\beta)\alpha+\beta)}b \\
&=am_\alpha b,
\end{align*}
which implies that $x=1m_\alpha t=f_{m_\alpha}(t)$ is a solution to equation
\eqref{fixed-eq-scalar} for $\tau=m_\alpha$ and $\sigma=m_\beta$, that is,
$(m_\alpha)_{m_\beta}=m_\alpha$ for all $\alpha\in[0,1]$ and $\beta\in(0,1]$. It was proved in
\cite{UYY} that if $\{m_\alpha\}_{\alpha\in[0,1]}$ is a regular continuous family of operator
means, then it is an interpolation family if and only if there exists an $r\in[-1,1]$ such that
$m_\alpha=\frak{p}_{\alpha,r}$ for all $\alpha\in[0,1]$. Therefore, we have, for every
$\alpha\in[0,1]$, $\beta\in(0,1]$ and $r\in[-1,1]$,
\begin{align}\label{power-mean-deform}
(\frak{p}_{\alpha,r})_{\frak{p}_{\beta,r}}=\frak{p}_{\alpha,r}.
\end{align}
In particular, for every $\alpha\in[0,1]$ and $\beta\in(0,1]$,
\begin{align*}
(\triangledown_\alpha)_{\triangledown_\beta}=\triangledown_\alpha,\qquad
(\#_\alpha)_{\#_\beta}=\#_\alpha,\qquad
(!_\alpha)_{\,!_\beta}=\,!_\alpha.
\end{align*}

The following proposition may be worth giving while it is not essentially new.

\begin{prop}\label{P-deform-inv}
For every symmetric operator mean $\sigma$ consider the following conditions:
\begin{itemize}
\item[(i)] $\sigma=\frak{p}_{1/2,r}$ for some $r\in[-1,1]$,
\item[(ii)] $(a\sigma b)\sigma(c\sigma d)=(a\sigma c)\sigma(b\sigma d)$ for all
$a,b,c,d\in(0,\infty)$,
\item[(iii)] $f_\sigma(x)\sigma f_\sigma(y)=f_\sigma(x\sigma y)$ for all $x,y\in(0,\infty)$,
\item[(iv)] $\sigma_\sigma=\sigma$.
\end{itemize}
Then we have (i) $\iff$ (ii) $\implies$ (iii) $\implies$ (iv).
\end{prop}

\begin{proof}
(i)$\implies$(ii) is immediately seen.

(ii)$\implies$(i).\enspace
Let $\{m_\alpha\}_{\alpha\in[0,1]}$ be the regular continuous family constructed in \cite{FK}
(as mentioned above) from $\sigma$ so that $m_{1/2}=\sigma$. It was shown in \cite[Theorem 1]{Fu}
that $\sigma$ satisfies condition (ii) if and only if $\{m_\alpha\}$ is an interpolation family.
Then it follows from \cite[Theorem 6]{UYY} that $\{m_\alpha\}=\{\frak{p}_{\alpha,r}\}_\alpha$
for some $r\in[-1,1]$. Hence $\sigma=m_{1/2}=\frak{p}_{1/2,r}$.

(ii)$\implies$(iii) is obvious by taking $a=c=1$ in (ii), as mentioned in \cite{Fu}.

(iii)$\implies$(iv).\enspace
For every $t>0$ let $x:=f_\sigma(t)$; then one has by (iii)
\begin{align*}
f_\sigma(1/x)f_\sigma\biggl({f_\sigma(t/x)\over f_\sigma(1/x)}\biggr)
&=f_\sigma(1/x)\sigma f_\sigma(t/x)=f_\sigma((1/x)\sigma(t/x)) \\
&=f_\sigma((1/x)f_\sigma(t))=f_\sigma(1)=1.
\end{align*}
This implies by Theorem \ref{T-main} that $\sigma=\sigma_\sigma$.
\end{proof}

\begin{problem}\label{Q-interpolated}\rm
Is there an operator mean $\sigma$ ($\ne\frak{l}$), apart from the weighted power means
$\frak{p}_{\alpha,r}$, such that $\sigma_\sigma=\sigma$? The condition $\sigma_\sigma=\sigma$
seems quite hard to hold unless $\sigma$ is in the family $\frak{p}_{\alpha,r}$. For example,
for $-1\le q\le2$ let $\sigma_q$ be a power difference mean with the representing function
$$
f_q(t)={q-1\over q}\cdot{1-t^q\over1-t^{q-1}};
$$
see \cite[Proposition 4.2]{HK} for the fact that $f_q\in\OM_{+,1}(0,\infty)$ when (and only
when) $-1\le q\le2$. A numerical computation verifies the failure of $\sigma_\sigma=\sigma$ for
$\sigma=\sigma_q$ except when $q=-1,1/2,2$ (the cases of the harmonic, geometric and arithmetic
means).
\end{problem}

\section{Multivariate means}

The aim of this section is to generalize our previous discussions on deformation of operator
means to multivariate means of operators. We here assume, unless otherwise stated, that $\cH$ is
a general Hilbert space, whichever finite-dimensional and infinite-dimensional. In what follows,
we consider an $n$-variable ($n\ge2$) mean of operators
\begin{align*}
M:(B(\cH)^{++})^n\ \longrightarrow\ B(\cH)^{++}
\end{align*}
having the following properties:
\begin{itemize}
\item[(A)] \emph{Joint monotonicity:} If $A_j,B_j\in B(\cH)^{++}$ and $A_j\le B_j$ for
$1\le j\le n$, then $M(A_1,\dots,A_n)\le M(B_1,\dots,B_n)$.
\item[(B)] \emph{Homogeneity:} $M(\alpha A_1,\dots,\alpha A_n)=\alpha M(A_1,\dots,A_n)$ for
every $A_j\in B(\cH)^{++}$ and $\alpha>0$.
\item[(C)] \emph{Monotone continuity:} If $A_{j,k}\in B(\cH)^{++}$, $k\in\bN$, and
$A_{j,k}\searrow A_j\in B(\cH)^{++}$ as $k\to\infty$ for $1\le j\le n$, then
$M(A_{1,k},\dots,A_{n,k})\to M(A_1,\dots,A_n)$ in the strong operator topology. The same
holds if $A_{j,k}\nearrow A_j\in B(\cH)^{++}$ as $k\to\infty$ for $1\le j\le n$.
\item[(D)] \emph{Normalization:} $M(I,\dots,I)=I$.
\end{itemize}

The above properties will always be assumed as minimal requirements for our arguments below on
the deformation of $M$ by operator means. Unlike ($2$-variable) operator means treated in
Sections 2 and 3, we fix the domain of multivariate means to $B(\cH)^{++}$ and do not
consider to extend them to $B(\cH)^+$. In the sequel, we will thus use a simpler notation $\bP$
in place of $B(\cH)^{++}$. We denote by $\bw=(w_1,\dots,w_n)$ a probability weight vector, i.e.,
$w_j\ge0$ with $\sum_{j=1}^nw_j=1$. The most familiar examples of multivariate means are the
weighted arithmetic and harmonic means
\begin{align*}
\cA_\bw(A_1,\dots,A_n):=\sum_{j=1}^nw_jA_j,\qquad
\cH_\bw(A_1,\dots,A_n):=\Biggl(\sum_{j=1}^nw_jA_j^{-1}\Biggr)^{-1}.
\end{align*}
More substantial and recently the most studied examples are the multivariate extensions of
the weighted geometric mean and the weighted power means, as briefly described below.

\begin{example}\label{E-multi-geom}\rm
The weighted multivariate geometric mean $G_\bw(A_1,\dots,A_n)$, variously called the
\emph{Riemannian mean}, the \emph{Karcher mean} and the \emph{Cartan mean}, was introduced for
positive definite matrices by Moakher \cite{Mo} and by Bhatia and Holbrook \cite{BH} in a
Riemannian geometry approach, whose monotonicity property (A) was proved by Lawson and Lim
\cite{LL1}. A significant feature of $G_\bw$ is that it is determined as a unique positive
definite solution to the \emph{Karcher equation}
\begin{align}\label{Karcher}
\sum_{j=1}^nw_j\log X^{-1/2}A_jX^{-1/2}=0,\qquad X\in\bP.
\end{align}
This equation is also used to define $G_\bw(A_1,\dots,A_n)$ for infinite-dimensional positive
operators, see \cite{LL2,LL3}. Note that $G_\bw$ is indeed the extension of the weighted
operator mean $\#_\alpha$ as $G_{(1-\alpha,\alpha)}(A,B)=A\#_\alpha B$.

Although it is not explicitly mentioned in \cite{LL3}, one can easily see from the arguments
there that $G_\bw$ satisfies the monotone continuity (C). (In the finite-dimensional case, this
is obvious from the continuity of $G_\bw$ in the operator norm shown in \cite{LL3}.) Indeed,
assume that $A_{j,k}\searrow A_j\in\bP$ for $1\le j\le n$, and let
$X_k:=G_\bw(A_{1,k},\dots,A_{n,k})$. One can choose a $\delta\in(0,1)$ so that
$\delta I\le A_{j,k}\le\delta^{-1}I$ for all $j,k$. Then $\delta I\le X_k\le\delta^{-1}I$
as well for all $k$ and the monotonicity property implies that $X_k\searrow X_0$ for some
$X_0\in\bP$. Since $(A,Y)\mapsto\log Y^{-1/2}AY^{-1/2}$ is continuous in the strong operator
topology on $\delta I\le A,Y\le\delta^{-1}I$ (see \cite{LL3}, also Remark \ref{R-conti}), it
follows that $X_0$ satisfies the Karcher equation in \eqref{Karcher} for $A_j$, so
$X_0=G_\bw(A_1,\dots,A_n)$. The proof in the case $A_{j,k}\nearrow A_j$ is similar.

From the viewpoint of the fixed point method, it is worth noting that, for every $A_j\in\bP$ and
$0<r\le1$, $X=G_\bw(A_1,\dots,A_n)$ is a unique solution to the equation
\begin{align}\label{eq-multi-geom}
X=G_\bw(X\#_rA_1,\dots,X\#_rA_n),\qquad X\in\bP.
\end{align}
This was shown in \cite[Theorem 4]{HL2} in the setting of probability measures on the positive
definite matrices, but the same proof is valid in the infinite-dimensional case as well.
\end{example}

\begin{example}\label{E-multi-power}\rm
The multivariate weighted power means $P_{\bw,r}(A_1,\dots,A_n)$, where
$r\in[-1,1]\setminus\{0\}$, was introduced for positive definite matrices by Lim and P\'alfia
\cite{LP1}, which was extended to infinite-dimensional operators in \cite{LL2,LL3}. For
$A_j\in\bP$, $P_{\bw,r}(A_1,\dots,A_n)$ is defined as a unique solution to the equation for
$X\in\bP$
\begin{align}
X&=\cA_\bw(X\#_rA_1,\dots,X\#_rA_n)\qquad\,\mbox{for\quad$0<r\le1$}, \label{eq-multi-power1}\\
X&=\cH_\bw(X\#_{-r}A_1,\dots,X\#_{-r}A_n)\quad\mbox{for $-1\le r<0$}, \label{eq-multi-power2}
\end{align}
which are the extension of $\frak{p}_{\alpha,r}$ given in Examples \ref{ex-2} and \ref{ex-3} as
$P_{(1-\alpha,\alpha),r}(A,B)=A\,\frak{p}_{\alpha,r}B$. Clearly,
$P_{\bw,1}=\cA_\bw$ and $P_{\bw,-1}=\cH_\bw$. An important fact proved in
\cite{LP1,LL2,LL3} is that
\begin{align*}
\lim_{r\to0}P_{\bw,r}(A_1,\dots,A_n)=G_\bw(A_1,\dots,A_n)
\end{align*}
in the strong operator topology.

The joint monotonicity (A) of $P_{\bw,r}$ was given in \cite{LP1,LL2,LL3} by a fixed point method.
By an argument similar to that in Example \ref{E-multi-geom} one can show that $P_{\bw,r}$
satisfies the monotone continuity (C). Indeed, assume that $A_{j,k}\searrow A_j\in\bP$ for
$1\le j\le n$, and choose a $\delta\in(0,1)$ such that $\delta I\le A_{j,k}\le\delta^{-1}I$ for
all $j,k$. Let $X_k:=P_{\bw,r}(A_{1,k},\dots,A_{n,k})$; then $X_k\searrow X_0$ for some
$X_0\in\bP$. It then follows that that
\begin{align*}
\cA_\bw(X_k\#_rA_{1,k},\dots,X_k\#_rA_{n,k})
\ &\longrightarrow\ \cA_\bw(X_0\#_rA_1,\dots,X_0\#_rA_n), \\
\cH_\bw(X_k\#_{-r}A_{1,k},\dots,X_k\#_{-r}A_{n,k})
\ &\longrightarrow\ \cH_\bw(X_0\#_{-r}A_1,\dots,X_0\#_{-r}A_n)
\end{align*}
in the strong operator topology for $0<r\le1$ and $-1\le r<0$, respectively. Hence
$X_0=P_{\bw,r}(A_1,\dots,A_n)$. The proof when $A_{j,k}\nearrow A_j$ is similar.
\end{example}

All of the multivariate means $\cA_\bw$, $\cH_\bw$, $G_\bw$ and $P_{\bw,r}$ given so far satisfy
all properties (A)--(D). Additionally they satisfy, among others, the following properties (see
\cite{LL2,LL3,Ya} and references therein), which will be separately treated below.

\begin{itemize}
\item[(E)] \emph{Congruence invariance:} For every $A_j\in\bP$ and any invertible $C\in B(\cH)$,
\begin{align*}
CM(A_1,\dots,A_n)C^*=M(CA_1C^*,\dots,CA_nC^*).
\end{align*}
\item[(F)] \emph{Joint concavity:} For every $A_j,B_j\in\bP$ and $0<\lambda<1$,
\begin{align*}
&M(\lambda A_1+(1-\lambda)B_1,\dots,\lambda A_n+(1-\lambda)B_n) \\
&\qquad\ge\lambda M(A_1,\dots,A_n)+(1-\lambda)M(B_1,\dots,B_n).
\end{align*}
From homogeneity (B) this is equivalent to
\begin{align*}
M(A_1+B_1,\dots,A_n+B_n)\ge M(A_1,\dots,A_n)+M(B_1,\dots,B_n).
\end{align*}
\item[(G)] \emph{$\cA M\cH$ weighted mean inequalities:} With some weight vector $\bw$, for every
$A_j\in\bP$,
\begin{align*}
\cH_\bw(A_1,\dots,A_n)\le M(A_1,\dots,A_n)\le\cA_\bw(A_1,\dots,A_n).
\end{align*}
It is indeed known \cite{LP1,LL2,LL3} that
\begin{align*}
\cH_\bw\le P_{\bw,-s}\le P_{\bw,-r}\le G_\bw\le P_{\bw,r}\le P_{\bw,s}\le\cA_\bw
\quad\mbox{if\ \ $0<r<s<1$}.
\end{align*}
\end{itemize}

From now on, assume that $M:\bP^n\to\bP$ is an $n$-variable mean of operators satisfying (A)--(D)
stated at the beginning of the section, and let $\sigma_1,\dots,\sigma_n$ be operator means (in
the sense of Kubo and Ando) such that $\sigma_j\ne\frak{l}$ for any $j$. For given
$A_1,\dots,A_n\in\bP$ we consider the equation
\begin{align}\label{eq-fixed-M}
X=M(X\sigma_1A_1,\dots,X\sigma_nA_n),\qquad X\in\bP,
\end{align}
which generalizes \eqref{eq-multi-geom}, \eqref{eq-multi-power1} and \eqref{eq-multi-power2}.

The following $\dT$-inequality is well-known for $G_\bw$ and $P_{\bw,r}$ (see \cite{LL3}). The
inequality for $M$ easily follows from (A) and (B) similarly to the proof of Lemma
\ref{L-Thomp}\,(a).

\begin{lemma}\label{L-ineq-M}
For every $A_j,B_j\in\bP$,
\begin{align*}
\dT(M(A_1,\dots,A_n),M(B_1,\dots,B_n))\le\max_{1\le j\le k}\dT(A_j,B_j).
\end{align*}
In particular, this implies that $M(A_1,\dots,A_n)$ is continuous on $\bP^n$ in the operator norm.
\end{lemma}

\begin{lemma}\label{L-fixed-M}
For every $A_j\in\bP$ there exists a unique $X_0\in B(\cH)^{++}$ which satisfies
\eqref{eq-fixed-M}. Furthermore, we have:
\begin{itemize}
\item[(1)] If $Y\in\bP$ and $Y\ge M(Y\sigma_1A_1,\dots,Y\sigma_nA_n)$, then $Y\ge X_0$.
\item[(2)] If $Y'\in\bP$ and $Y'\le M(Y'\sigma_1A_1,\dots,Y'\sigma_nA_n)$, then $Y'\le X_0$.
\end{itemize}
\end{lemma}

\begin{proof}
The proof is similar to that of Lemmas \ref{L-fixed} and \ref{L-fixed-ineq}, by use of a map
$F$ from $\bP$ into itself defined by
\begin{align}\label{F-multi}
F(X):=M(X\sigma_1A_1,\dots,X\sigma_nX_n).
\end{align}
We only confirm the uniqueness of the solution here. Assume that $X_0,X_1\in\bP$ satisfies
\eqref{eq-fixed-M} and $X_0\ne X_1$. By Lemmas \ref{L-ineq-M} and \ref{L-Thomp}\,(b),
\begin{align*}
\dT(X_0,X_1)\le\max_{1\le j\le n}\dT(X_0\sigma_jA_j,X_1\sigma_jA_j)<\dT(X_0,X_1),
\end{align*}
a contradiction.
\end{proof}

For every $A_j\in\bP$, $1\le j\le n$, we write $M_{(\sigma_1,\dots,\sigma_n)}(A_1,\dots,A_n)$ for
the unique solution $X_0\in\bP$ to \eqref{eq-fixed-M} given in Lemma \ref{L-fixed-M}, and so we
have a map
\begin{align}\label{deform-M}
M_{(\sigma_1,\dots,\sigma_n)}:\bP^n\,\longrightarrow\,\bP.
\end{align}

\begin{lemma}\label{L-properties-M}
The map $M_{(\sigma_1,\dots,\sigma_n)}$ satisfies (A)--(D) as $M$ does the same.
\end{lemma}

\begin{proof}
The proof is similar to that of Lemma \ref{L-properties}, so we omit the details. Note here that
Lemma \ref{L-fixed-M} is used to show (A).
\end{proof}

\begin{remark}\label{R-multi-deform}\rm
When $\cH$ is finite-dimensional, we can prove the unique existence of the solution to equation
\eqref{eq-fixed-M} under (A), (B) and (D) without (C). To see this, choose a $\delta\in(0,1)$
such that $\delta I\le A_j\le\delta^{-1}I$ for $1\le j\le n$, and let $\Sigma_\delta:=\{X\in\bP:
\delta I\le X\le\delta^{-1}I\}$ which is a compact convex subset of the $d^2$-dimensional
Euclidean space $B(\cH)^{sa}$, the space of self-adjoint $X\in B(\cH)$, where $d:=\dim\cH$. It
follows from (A), (B) and (D) of $M$ that if $X\in\Sigma_\delta$ then $F(X)$ given in
\eqref{F-multi} is in $\Sigma_\delta$. Since $F$ is continuous by Lemma \ref{L-ineq-M} where
assumption (C) is unnecessary, the existence of the solution follows from Brouwer's fixed point
theorem, and its uniqueness was shown above. However, under this situation without (C) we are not
able to show that the resulting map \eqref{deform-M} satisfies (A) and (B).
\end{remark}

We call $M_{(\sigma_1,\dots,\sigma_n)}$ the \emph{deformed mean} of operators from $M$ by
$(\sigma_1,\dots,\sigma_n)$. When $\sigma_1=\dots=\sigma_n=\sigma$, we simply write $M_\sigma$
and call it the deformed mean from $M$ by $\sigma$. The deformed operator mean $\tau_\sigma$
discussed in Sections 2 and 3 is the special case where $M=\tau$ is a ($2$-variable) operator
mean and $\sigma_1=\sigma_2=\sigma$. The multivariate weighted power means $P_{\bw,r}$ in
Example \ref{E-multi-power} are typical examples of deformed means as
\begin{align*}
P_{\bw,r}=(\cA_\bw)_{\#_r},\quad P_{\bw,-r}=(\cH_\bw)_{\#_r}\quad\mbox{for\quad$0<r\le1$}.
\end{align*}
From Example \ref{E-multi-geom} note also that
\begin{align*}
G_\bw=(G_\bw)_{\#_r}\quad\mbox{for\quad$0<r\le1$}.
\end{align*}

The \emph{adjoint mean} $M^*$ of $M$ is defined by
\begin{align*}
M^*(A_1,\dots,A_n):=M(A_1^{-1},\dots,A_n^{-1})^{-1},\qquad A_j\in\bP.
\end{align*}
It is immediate to verify that $M^*$ satisfies (A)--(D) as well. The mean $M$ is said to be
\emph{self-adjoint} if $M=M^*$. Note that for multivariate means the term ``dual" is rather used
for $M^*$ as in \cite{LP1,LL2,LL3}, but we prefer to use the term ``adjoint" in accordance with
the case of operator means.

\begin{prop}\label{P-properties-M}
\begin{itemize}
\item[(1)] $M_{\frak{r}}=M$.
\item[(2)] Let $\widehat M$ be an $n$-variable mean with (A)--(D) and $\widehat\sigma_j$,
$1\le j\le n$, be operator means. If $M\le\widehat M$ and $\sigma_j\le\widehat\sigma_j$ for
$1\le j\le n$, then
$M_{(\sigma_1,\dots,\sigma_n)}\le\widehat M_{(\widehat\sigma_1,\dots,\widehat\sigma_n)}$.
\item[(3)] $(M_{(\sigma_1,\dots,\sigma_n)})^*=(M^*)_{(\sigma_1^*,\dots,\sigma_n^*)}$. Hence, if
$M$ and $\sigma_j$ are self-adjoint, then so is $M_{(\sigma_1,\dots,\sigma_n)}$.
\item[(4)] For any permutation $\pi$ on $\{1,\dots,k\}$ define
\begin{align*}
M^\pi(A_1,\dots,A_n):=M(A_{\pi^{-1}(1)},\dots,A_{\pi^{-1}(k)}).
\end{align*}
Then $(M_{(\sigma_1,\dots,\sigma_n)})^\pi=(M^\pi)_{(\sigma_{\pi(1)},\dots,\sigma_{\pi(k)})}$.
Hence, if $M$ is permutation invariant, then so is $M_\sigma$ for any operator mean $\sigma$.
\item[(5)] If $M$ satisfies (E), then $M_{(\sigma_1,\dots,\sigma_n)}$ does the same.
\item[(6)] If $M$ satisfies (F), then $M_{(\sigma_1,\dots,\sigma_n)}$ does the same.
\item[(7)] Assume that $M$ satisfies (G) with a weight vector $\bw$, and let
\begin{align*}
\widehat\bw:=\Biggl(\sum_{j=1}^nw_j\alpha_j\Biggr)^{-1}(w_1\alpha_1,\dots,w_n\alpha_n),
\end{align*}
where $\alpha_j:=f_{\sigma_j}'(1)$, $1\le j\le n$. Then $M_{(\sigma_1,\dots,\sigma_n)}$
satisfies (G) with the weight vector $\widehat\bw$. In particular, for any operator mean
$\sigma$, $M_\sigma$ satisfies (G) with the same $\bw$.
\end{itemize}
\end{prop}

\begin{proof}
(1) is obvious. The proofs of (2) and (3) are similar to those of (3) and (4), respectively, of
Proposition \ref{P-properties}.

(4) immediately follows since
\begin{align*}
M^\pi(X\sigma_{\pi(1)}A_1,\dots,X\sigma_{\pi(n)}A_n)
=M(X\sigma_1A_{\pi^{-1}(1)},\dots,X\sigma_nA_{\pi^{-1}(n)}).
\end{align*}

(5)\enspace
The proof is similar to that of Lemma \ref{L-properties}\,(ii).

(6)\enspace
Let $X_0:=M_{(\sigma_1,\dots,\sigma_n)}(A_1,\dots,A_n)$ and
$Y_0:=M_{(\sigma_1,\dots,\sigma_n)}(B_1,\dots,B_n)$. Since, by \cite[Theorem 3.5]{KA},
\begin{align*}
(X_0+Y_0)\sigma_j(A_j+B_j)\ge(X_0\sigma_jB_j)+(Y_0\sigma_jB_j),
\end{align*}
we have, by (A) and (F) of $M$,
\begin{align*}
&M((X_0+Y_0)\sigma_1(A_1+B_1),\dots,(X_0+Y_0)\sigma_n(A_n+B_n)) \\
&\qquad\ge M(X_0\sigma_1A_1,\dots,X_0\sigma_nA_n)+M(Y_0\sigma_1B_1,\dots,Y_0\sigma_nB_n)
=X_0+Y_0,
\end{align*}
which implies by Lemma \ref{L-fixed-M}\,(2) that
\begin{align*}
M_{(\sigma_1,\dots,\sigma_n)}(A_1+B_1,\dots,A_n+B_n)\ge X_0+Y_0,
\end{align*}
as required.

(7)\enspace
Since $!_{\alpha_j}\le\sigma_j\le\triangledown_{\alpha_j}$, it follows from the assertion (2)
above that
\begin{align*}
(\cH_\bw)_{(!_{\alpha_1},\dots,!_{\alpha_n})}\le M_{(\sigma_1,\dots,\sigma_n)}
\le(\cA_\bw)_{(\triangledown_{\alpha_1},\dots,\triangledown{\alpha_n})}.
\end{align*}
The equation
\begin{align*}
X=\cA_\bw(X\triangledown_{\alpha_1}A_1,\dots,X\triangledown_{\alpha_n}A_n)
\end{align*}
is easily solved as
\begin{align*}
X=\Biggl(\sum_{j=1}^nw_j\alpha_j\Biggr)^{-1}
\sum_{j=1}^nw_j\alpha_jA_j=\cA_{\widehat\bw}(A_1,\dots,A_n),
\end{align*}
while the equation
\begin{align*}
X=\cH_\bw(X\,!_{\alpha_1}A_1,\dots,X\,!_{\alpha_n}A_n)
\end{align*}
is solved as $X=\cH_{\widehat\bw}(A_1,\dots,A_n)$. (The latter also follows from the former and
the assertion (3).) Therefore, $M_{(\sigma_1,\dots,\sigma_n)}$ satisfies (G) with the weight
vector $\widehat\bw$.
\end{proof}

\begin{prop}\label{P-mono-conti-M}
For each $j=1,\dots,n$ let $\sigma_j$ and $\sigma_{j,k}$, $k\in\bN$, be operator means with
$\sigma_j,\sigma_{j,k}\ne\frak{l}$, and assume that $\sigma_{j,k}\searrow\sigma_j$, that is,
$f_{\sigma_{j,k}}(t)\searrow f_{\sigma_j}(t)$ for any $t\in(0,\infty)$ as $k\to\infty$. Then for
every $A_j\in\bP$,
\begin{align*}
M_{(\sigma_{1,k},\dots,\sigma_{n,k})}(A_1\dots,A_n)
\ \searrow\ M_{(\sigma_1,\dots,\sigma_n)}(A_1,\dots,A_n)\quad\mbox{as $k\to\infty$}
\end{align*}
in the strong operator topology. The upward convergence holds similarly when
$\sigma_{j,k}\nearrow\sigma_j$ as $k\to\infty$ for each $j$.
\end{prop}

\begin{proof}
Let $X_k:=M_{(\sigma_{1,k},\dots,\sigma_{n,k})}(A_1,\dots,A_n)$ for each $k\in\bN$. It follows
from Lemma \ref{P-properties-M}\,(2) that $X_1\le X_2\le\cdots$. Choose a $\delta>0$ such that
$A_j\ge\delta I$ for $1\le j\le n$. Since
$\delta I\le M((\delta I)\sigma_{1,k}A_1,\dots,(\delta I)\sigma_{n,k}A_n)$, Lemma
\ref{L-fixed-M}\,(2) gives $X_k\ge\delta I$ for all $k$. Therefore, $X_k\searrow X_0$ for some
$X_0\in\bP$. We may now prove that $X_0=M_{(\sigma_1,\dots,\sigma_n)}(A_1,\dots,A_n)$, that is,
$X_0=M(X_0\sigma_1A_1,\dots,X_0\sigma_nA_n)$. From (C) it suffices to show that
$X_k\sigma_{j,k}A_j\searrow X_0\sigma_jA_j$ as $k\to\infty$ for every $j=1,\dots,n$. But this is
easy to verify since $X_k\searrow A_0$ and $\sigma_{j,k}\searrow\sigma_j$. The proof is similar
when $\sigma_{j,k}\nearrow\sigma_j$,
\end{proof}

In the following we present the multivariate versions of Proposition \ref{P-conti} and Theorem
\ref{T-tau-boundary}, while $\cH$ is assumed here to be finite-dimensional.

\begin{prop}\label{P-conti-M}
Assume that $\cH$ is finite-dimensional. For each $j=1,\dots,n$ let $\sigma_j,\sigma_{j,k}$,
$k\in\bN$, be operator means with $\sigma_j,\sigma_{j,k}\ne\frak{l}$, and assume that
$\sigma_{j,k}\to\sigma_j$ properly as $k\to\infty$. Then for every $A_j\in\bP$,
\begin{align*}
M_{(\sigma_{1,k},\dots,\sigma_{n,k})}(A_1,\dots,A_n)
\ \longrightarrow\ M_{(\sigma_1,\dots,\sigma_n)}(A_1,\dots,A_n)
\end{align*}
in the operator norm.
\end{prop}

\begin{proof}
Choose a $\delta\in(0,1)$ such that $\delta I\le A_j\le\delta^{-1}I$ for $1\le j\le n$. Let
$X_0:=M_{(\sigma_1,\dots,\sigma_n)}(A_1,\dots,A_n)$ and
$X_k:=M_{(\sigma_{1,k},\dots,\sigma_{n,k})}(A_1,\dots,A_n)$; then by Lemma \ref{L-fixed-M}\,(1)
and (2), $\delta I\le X_k\le\delta^{-1}I$ as well for all $k$. From the finite-dimensionality
assumption, note that $\{X\in B(\cH)^{++}:\delta I\le X\le\delta^{-1}I\}$ is compact in the
operator norm. Hence, to prove that $X_k\to X_0$ in the operator norm, it suffices to show that
$X_0$ is a unique limit point of $\{X_k\}$. By replacing $\{X_k\}$ by a subsequence we may
assume (for notational brevity) that $\{X_k\}$ itself converges to a $Y_0\in\bP$. For each
$j=1,\dots,n$ we have
\begin{align*}
\|X_k\sigma_{j,k}A_j-Y_0\sigma_jA_j\|
&=\|A_j\sigma_{j,k}'X_k-A_j\sigma_j'Y_0\| \\
&\le\|A_j\|\,\|f_{\sigma_{j,k}'}(A_j^{-1/2}X_kA_j^{-1/2})
-f_{\sigma_j'}(A_j^{-1/2}Y_0A_j^{-1/2})\| \\
&\le\|A_j\|\,\|f_{\sigma_{j,k}'}(A_j^{-1/2}X_kA_j^{-1/2})
-f_{\sigma_j'}(A_j^{-1/2}X_kA_j^{-1/2})\| \\
&\quad+\|A_j\|\,\|f_{\sigma_j'}(A_j^{-1/2}X_kA_j^{-1/2})
-f_{\sigma_j'}(A_j^{-1/2}Y_0A_j^{-1/2})\|,
\end{align*}
where $\sigma_j',\sigma_{j,k}'$ are the transposes of $\sigma_j,\sigma_{j,k}$. Since
$f_{\sigma_{j,k}'}(t)=tf_{\sigma_{j,k}}(t^{-1})\to tf_{\sigma_j}(t^{-1})=f_{\sigma_j'}(t)$ for
any $t\in(0,\infty)$ and $\delta A_j^{-1}\le A_j^{-1/2}X_kA_j^{-1/2}\le\delta^{-1}A_j^{-1}$ for
all $k$, it follows from Lemma \ref{L-OM-conv} that
\begin{align*}
\|f_{\sigma_{j,k}'}(A_j^{-1/2}X_kA_j^{-1/2})-f_{\sigma_j'}(A_j^{-1/2}X_kA_j^{-1/2})\|
\ \longrightarrow\ 0\quad\mbox{as $k\to\infty$},
\end{align*}
so that $\|X_k\sigma_{j,k}A_j-Y_0\sigma_jA_j\|\to0$. Therefore, by Lemma \ref{L-ineq-M},
\begin{align*}
X_k=M(X_k\sigma_{1,k}A_1,\dots,X_k\sigma_{n,k}A_n)
\ \longrightarrow\ M(Y_0\sigma_1A_1,\dots,Y_0\sigma_nA_n)
\end{align*}
in the operator norm. This implies that $Y_0=M(Y_0\sigma_1A_1,\dots,Y_0\sigma_nA_n)$, and hence
$Y_0=X_0$ follows.
\end{proof}

\begin{thm}\label{T-M-boundary}
Assume that $\cH$ is finite-dimensional and $M$ satisfies (G) with a weight vector $\bw$. Then
for every $A_j\in\bP$,
\begin{align*}
&M_{(\frak{p}_{s,r_1},\dots,\frak{p}_{s,r_n})}(A_1,\dots,A_n) \\
&\quad\longrightarrow
\begin{cases}\,(\cA_{\widehat\bw})_{(\#_{r_1},\dots,\#_{r_n})}(A_1,\dots,A_n)
& \text{if $r_1,\dots,r_n\in(0,1]$}, \\
\,(\cH_{\widehat\bw})_{(\#_{-r_1},\dots,\#_{-r_n})}(A_1,\dots,A_n)
& \text{if $r_1,\dots,r_n\in[-1,0)$}, \\
\,G_\bw(A_1,\dots,A_n) & \text{if $r_1=\dots=r_n=0$},
\end{cases}
\end{align*}
in the operator norm as $s\searrow0$ with $s\in(0,1]$, where
\begin{align*}
\widehat\bw:=\begin{cases}
\,\Bigl(\sum_{j=1}^n{w_j\over r_j}\Bigr)^{-1}\Bigl({w_1\over r_1},\dots,{w_n\over r_n}\Bigr)
& \text{if $r_1,\dots,r_n\in(0,1]$}, \\
\,\Bigl(\sum_{j=1}^n{w_j\over-r_j}\Bigr)^{-1}\Bigl({w_1\over-r_1},\dots,{w_n\over-r_n}\Bigr)
& \text{if $r_1,\dots,r_n\in[-1,0)$}.
\end{cases}
\end{align*}
\end{thm}

\begin{proof}
In view of property (G) and Proposition \ref{P-properties-M}\,(2) we may prove the result for
$M=\cA_\bw$ and for $M=\cH_\bw$. Choose a $\delta\in(0,1)$ such that
$\delta I\le A_j\le\delta^{-1}I$ for $1\le j\le n$. First, let us prove the case $M=\cA_\bw$, and
let either $r_1,\dots,r_n\in(0,1]$ or $r_1,\dots,r_n\in[-1,0)$. For $s\in(0,1]$ let
$X_s:=(\cA_\bw)_{(\frak{p}_{s,r_1},\dots,\frak{p}_{s,r_n})}(A_1,\dots,A_n)$. Since
$I=(\cA_\bw)_{(\frak{p}_{s,r_1},\dots,\frak{p}_{s,r_n})}(X_s^{-1/2}A_1X_s^{-1/2},\dots,
X_s^{-1/2}A_nX_s^{-1/2})$ by Lemma \ref{L-properties-M}, we have
\begin{align}\label{eq-X-fixed}
I=\sum_{j=1}^nw_j\bigl[(1-s)I+s(X_s^{-1/2}A_jX_s^{-1/2})^{r_j}\bigr]^{1/r_j}.
\end{align}
By taking account of $\{X_s\}_{s\in(0,1]}$ being in the compact set
$\{X\in\bP:\delta I\le X\le\delta^{-1}I\}$ (due to the finite-dimensionality of $\cH$), let $Y_0$
be any limit point of $X_s$ as $s\searrow0$, so $X_k:=X_{s_k}\to Y_0$ for some sequence
$s_k\in(0,1]$, $s_k\searrow0$. By \eqref{eq-X-fixed} one can write, as $k\to\infty$,
\begin{align*}
I&=\sum_{j=1}^nw_j\biggl[I+{s_k\over r_j}\bigl\{(X_k^{-1/2}A_jX_k^{-1/2})^{r_j}-I\bigr\}
+o(s_k)\biggr] \\
&=I+s_k\sum_{j=1}^n{w_j\over r_j}\bigl\{(X_k^{-1/2}A_jX_k^{-1/2})^{r_j}-I\bigr\}+o(s_k),
\end{align*}
so that
\begin{align}\label{eq-X-fixed2}
0=\sum_{j=1}^n{w_j\over r_j}\bigl\{(X_k^{-1/2}A_jX_k^{-1/2})^{r_j}-I\bigr\}+{o(s_k)\over s_k}.
\end{align}
When $r_1,\dots,r_n\in(0,1]$, this means that
\begin{align*}
X_k=\sum_{j=1}^n\widehat w_j(X_k\#_{r_j}A_j)+{o(s_k)\over s_k}.
\end{align*}
Letting $k\to\infty$ gives $Y_0=\sum_{j=1}^n\widehat w_j(Y_0\#_{r_j}A_j)$, and hence $Y_0=X_0$
follows, where $X_0:=(\cA_{\hat\bw})_{(\#_{r_1},\dots,\#_{r_n})}(A_1,\dots,A_n)$. Since $\{X_s\}$
has s unique limit point $X_0$, we find that $X_s\to X_0$ as $s\searrow0$.

On the other hand, when $r_1,\dots,r_n\in[-1,0)$, \eqref{eq-X-fixed2} is rewritten as
\begin{align*}
0=\sum_{j=1}^n{w_j\over-r_j}\bigl\{(X_k^{1/2}A_j^{-1}X_k^{1/2})^{-r_j}-I\bigr\}
+{o(s_k)\over s_k},
\end{align*}
that is,
\begin{align*}
X_k^{-1}=\sum_{j=1}^n\widehat w_j(X_k^{-1}\#_{-r_j}A_j^{-1})+{o(s_k)\over s_k}.
\end{align*}
Letting $k\to\infty$ gives
\begin{align*}
Y_0^{-1}=\sum_{j=1}^n\widehat w_j(Y_o^{-1}\#_{-r_j}A_J^{-1})
=\sum_{j=1}^n\widehat w_j(Y_0\#_{-r_j}A_j)^{-1},
\end{align*}
which means that $Y_0=X_0$, where
$X_0:=(\cH_{\widehat w})_{(\#_{-r_1},\dots,\#_{-r_n})}(A_1,\dots,A_n)$. Therefore, $X_s\to X_0$
as $s\searrow0$.

Next, assume that $r_1=\dots=r_n=0$, and let $X_s:=(\cA_\bw)_{\#_s}(A_1,\dots,A_n)$ (recall the
convention $\frak{p}_{s,0}=\#_s$ in \eqref{convention}). Then
\begin{align*}
I=\sum_{j=1}^nw_j(I\#_sX_s^{_1/2}A_jX_s^{-1/2})
=\sum_{j=1}^nw_j(X_s^{-1/2}A_jX_s^{-1/2})^s.
\end{align*}
Let $Y_0$ be any limit point of $X_s$ as $s\searrow0$, so $X_k:=X_{s_k}\to Y_0$ for some sequence
$s_k\in(0,1]$, $s_k\searrow0$. Since, as $k\to\infty$,
\begin{align*}
I&=\sum_{j=1}^nw_j\exp\bigl(s_k\log X_k^{-1/2}A_jA_k^{-1/2}) \\
&=\sum_{j=1}^nw_j\bigl[I+s_k\log X_k^{-1/2}A_jX_k^{-1/2}+o(s_k)\bigr] \\
&=I+s_k\sum_{j=1}^nw_j\log X_k^{-1/2}A_jX_s^{-1/2}+o(s_k),
\end{align*}
one has
\begin{align*}
\sum_{j=1}^nw_j\log X_k^{-1/2}A_jX_k^{-1/2}+{o(s_k)\over s_k}=0.
\end{align*}
Letting $k\to\infty$ gives the Karcher equation $\sum_{j=1}^nw_j\log Y_0^{-1/2}A_jY_0^{-1/2}=0$,
implying that $Y_0=G_\bw(A_1,\dots,A_n)$. Therefore, $G_\bw(A_1,\dots,A_n)$ is a unique limit
point of $X_s$, so $X_s\to G_\bw(A_1,\dots,A_n)$ as $s\searrow0$.

Second, let us treat the case $M=\cH_\bw$. The proof in this case is similar to the above case.
When either $r_j\in(0,1]$ or $r_j\in[-1,0)$, let
$X_s:=(\cH_\bw)_{(\frak{p}_{s,r_1},\dots,\frak{P}_{s,r_n})}(A_1,\dots,A_n)$ for $s\in(0,1]$. Then
\eqref{eq-X-fixed} is replaced with
\begin{align*}
I=\sum_{j=1}^nw_j\bigl[(1-s)I+s(X_s^{-1/2}A_jX_s^{-1/2})^{r_j}\bigr]^{-1/r_j},
\end{align*}
which yields the same equation as \eqref{eq-X-fixed2}. Hence the remaining proof is the same as
before. When $r_1=\dots=r_n=0$, the proof is similar to the above case as well.
\end{proof}

Similarly to $\tau_{s,r}$ for an operator mean $\tau$ defined in \eqref{tau-s,r} and
\eqref{tau-s,r-boundary}, for $M$ satisfying (G) with a weight vector $\bw$, we now define
the family $M_{s,r}$ of $n$-variable means of operators with two parameters $s\in[0,1]$ and
$r\in[-1,1]$ by
\begin{align*}
M_{s,r}:=\begin{cases}
M_{\frak{p}_{s,r}} & s\in(0,1],\ r\in[-1,1], \\
P_{\bw,r}, & s=0,\ r\in[-1,1]\setminus\{0\}, \\
G_\bw, & s=0,\ r=0.
\end{cases}
\end{align*}
Then we have

\begin{thm}\label{T-M-r,s}
Assume that $\cH$ is finite-dimensional and $M$ satisfies (G) with a weight vector $\bw$. Then
the family $M_{s,r}$ is continuous (in the sense of pointwise operator norm convergence) in
$s\in[0,1]$ and $r\in[-1,1]$.
\end{thm}

\begin{proof}
The continuity of $M_{s,r}$ on $(s,r)\in(0,1]\times[-1,1]$ is a consequence of Proposition
\ref{P-conti-M}. Hence we may show the continuity at $(0,r)$ where $r\in[-1,1]$. Theorem
\ref{T-M-boundary} in particular says that $M_{s_k,r}\to M_{0,r}$ as $s_k\to0$ with
$s_k\in(0,1]$ for any fixed $r\in[-1,1]$. But, the proof of Theorem \ref{T-M-boundary} can
slightly be modified to show that $M_{s_k,r_k}\to M_{0,r}=P_{\bw,r}$ as $s_k\to0$ and
$r_k\to r\ne0$ where $s_k\in(0,1]$ and $r_k\in[-1,1]$. Moreover, it is known \cite{LP1,LL2,LL3}
that $M_{0,r}=P_{\bw,r}\to M_{0,0}=G_\bw$ as $r\to0$. Thus it remains to show that
$M_{s_k,r_k}\to M_{0,0}=G_\bw$ as $s_k\to0$ and $r_k\to0$ with $s_k\in(0,1]$ and
$r_k\in[-1,1]\setminus\{0\}$. For this, as in the proof of Theorem \ref{T-M-boundary} we may
prove the convergence for $M=\cA_\bw$ and for $M=\cH_\bw$. For $M=\cA_\bw$ let
$X_k:=(\cA_\bw)_{\frak{p}_{s_k,r_k}}(A_1,\dots,A_n)$; then
\begin{align}\label{eq-X-fixed3}
I=\sum_{j=1}^nw_j\bigl[(1-s_k)I+s_k(X_k^{-1/2}A_jX_k^{-1/2})^{r_k}\bigr]^{1/r_k}.
\end{align}
As in the proof (the part of $r_k,s_k\to0$) of Theorem \ref{T-tau-boundary} by replacing $1/x_k$
with the operator $X_k^{-1/2}A_jA_k^{-1/2}$, one can prove that
\begin{align*}
&\bigl[I+s_k\bigl\{(X_k^{-1/2}A_jX_k^{-1/2})^{r_k}-I\bigr\}\bigr]^{1/r_k} \\
&\qquad=1+s_k\log X_k^{-1/2}A_jX_k^{-1/2}+o(s_k)\quad\mbox{as $k\to\infty$}.
\end{align*}
Inserting this into \eqref{eq-X-fixed3} gives
\begin{align*}
\sum_{j=1}^nw_j\log X_k^{-1/2}A_jX_k^{-1/2}+{o(s_k)\over s_k}=0
\quad\mbox{as $k\to\infty$},
\end{align*}
from which $X_k\to G_\bw(A_1,\dots,A_n)$ follows as in the proof (the part of $r_1=\dots=r_n=0$)
of Theorem \ref{T-M-boundary}. The proof for the case $M=\cH_\bw$ is similar.
\end{proof}

\begin{problem}\label{Q-infinite-dim}\rm
Properties of $2$-variable operator means reduce to that of operator monotone functions on
$(0,\infty)$ due to Kubo and Ando \cite{KA}, however this is not the case for multivariate
means of operators. Therefore, in the proofs of Proposition \ref{P-conti-M} and Theorems
\ref{T-M-boundary} and \ref{T-M-r,s} we have used the compactness argument, the reason why $\cH$
is assumed finite-dimensional. It seems that we have to find a new technique to prove those
results in the infinite-dimensional setting, while they are likely to hold.
\end{problem}

\begin{example}\label{E-M-deform1}\rm
From Theorem \ref{T-M-r,s} one has a lot of one parameter continuous families of $n$-variable
means of operators satisfying (A)--(D) and (G) (at least $\cH$ is finite-dimensional). For example,
when $M=\cA_\bw$ or $G_\bw$ or $\cH_\bw$, $\{M_{s,r}\}_{-1\le r\le1}$ is a continuous family of
means interpolating $M_{!_s}$ ($r=-1$), $M_{\#_s}$ ($r=0$) and $M_{\triangledown_s}$ ($r=1$)
for each $s\in(0,1]$, and $\{M_{s,r}\}_{0\le s\le1}$ is a continuous family of means joining
$P_{\bw,r}$ ($s=0$) and $M$ ($s=1$) for each $r\in[-1,1]\setminus\{0\}$ (also joining $G_\bw$
and $M$ when $r=0$). In particular, for $\cA$ and $\cH$ with weight $\bw=(1/n,\dots,1/n)$,
$X=\cA_{\,!}(A_1,\dots,A_n)$ and $X=\cH_\triangledown(A_1,\dots,A_n)$ are respectively the solutions
to
\begin{align*}
X={1\over n}\sum_{j=1}^nX\,!A_j,\qquad X^{-1}={1\over n}\sum_{j=1}^nX^{-1}\,!A_j^{-1},
\qquad X\in\bP.
\end{align*}
Unlike $\triangledown_{\,!}=\,!_\triangledown=\#$ in \eqref{harm-arith-geom}, $\cA_{\,!}$ and
$\cH_\triangledown=(\cA_{\,!})^*$ are not $G$ ($=G_\bw$ with $\bw=(1/n,\dots,1/n)$). In fact,
even for scalars $a_1,a_2,a_3>0$, $x=\cH_\triangledown(a_1,a_2,a_3)$ is a solution to
\begin{align*}
x^3+{a_1+a_2+a_3\over3}\,x^2-{a_1a_2+a_2a_3+a_3a_1\over3}\,x-a_1a_2a_3=0,\qquad x>0,
\end{align*}
which is not equal to $G(a_1,a_2,a_3)=(a_1a_2a_3)^{1/3}$.
\end{example}

\begin{example}\label{E-M-deform2}\rm
Since equation \eqref{eq-multi-geom} determines $X=G_\bw(A_1,\dots,A_n)$, we have
\begin{align*}
(G_\bw)_{\#_\alpha}=G_\bw,\qquad\alpha\in(0,1].
\end{align*}
Extending \eqref{power-mean-deform} we furthermore have
\begin{align*}
(P_{\bw,r})_{\frak{p}_{\alpha,r}}=P_{\bw,r}
\end{align*}
for every weight vector $\bw$, $\alpha\in(0,1]$ and $r\in[-1,1]\setminus\{0\}$. To see this, let
$X_0:=P_{\bw,r}(A_1,\dots,A_n)$, which is determined by the equation
\begin{align*}
I=\sum_{j=1}^nw_j(X_0^{-1/2}A_jX_0^{-1/2})^r.
\end{align*}
For every $\alpha\in(0,1]$ this is equivalent to
\begin{align*}
I&=\sum_{j=1}^nw_j\bigl[(1-\alpha)I+\alpha(X_0^{-1/2}A_jX_0^{-1/2})^r\bigr] \\
&=\sum_{j=1}^nw_j(I\,\frak{p}_{\alpha,r}(X_0^{-1/2}A_jX_0^{-1/2}))^r,
\end{align*}
which means that
\begin{align*}
I=P_{\bw,r}\bigl(I\,\frak{p}_{\alpha,r}(X_0^{-1/2}A_1X_0^{-1/2}),\dots,
I\,\frak{p}_{\alpha,r}(X_0^{-1/2}A_nX_0^{-1/2})\bigr).
\end{align*}
Therefore, $X_0=P_{\bw,r}(X_0\,\frak{p}_{\alpha,r}A_1,\dots,X_0\,\frak{p}_{\alpha,r}A_n)$, that is,
$X_0=(P_{\bw,r})_{\frak{p}_{\alpha,r}}(A_1,\dots,A_n)$.
\end{example}

\section{Final remark}

Assume that $\cH$ is finite-dimensional, and let $\bP$ denote $B(\cH)^{++}$ as in Section 4. Let
$\cP(\bP)$ be the set of all Borel probability measures on $\bP$. Apart from the Thompson metric
$\dT$, the \emph{Riemannian trace metric} on $\bP$ is given as
$\delta(A,B):=\|\log A^{-1/2}BA^{-1/2}\|_2$, where $\|\cdot\|_2$ is the Hilbert-Schmidt norm.
Then the Polish space $(\bP,\delta)$ is a typical \emph{NPC} (nonpositive curvature) space.
Therefore, the general theory of NPC spaces (see \cite{St}) is applied to define the
\emph{Cartan barycenter} $G(\mu)$ of $\mu\in\cP^1(\bP)$ as
\begin{align*}
G(\mu)=\mathop{\mathrm{arg\,min}}_{Z\in\bP}
\int_\bP\bigl[\delta^2(Z,X)-\delta^2(Y,X)\bigr]\,d\mu(X)
\end{align*}
(independently of the choice of a fixed $Y\in\bP$). Here, $\cP^1(\bP)$ is the set of
$\mu\in\cP(\bP)$ with finite first moment, i.e., for some (hence all) $Y\in\bP$,
$\int_\bP\delta(X,Y)\,d\mu(X)<\infty$. A fundamental property of $G(\mu)$ is the contraction
$\delta(G(\mu),G(\nu))\le d_1^W(\mu,\nu)$, $\mu,\nu\in\cP^1(\bP)$, where $d_1^W$ is the
\emph{$1$-Wasserstein distance} on $\cP^1(\bP)$. The $G(\mu)$ extends the multivariate
geometric mean (or the Karcher mean) $G_\bw$ in Example \ref{E-multi-geom} as
$G(\mu)=G_\bw(A_1,\dots,A_n)$ if $\mu=\sum_{j=1}^nw_j\delta_{A_j}$. In fact, it is known
\cite{HL1} that $G(\mu)$ is the unique solution to the Karcher equation
\begin{align}\label{geom-prob}
\int_\bP\log X^{-1/2}AX^{-1/2}\,d\mu(A)=0,\qquad X\in\bP.
\end{align}

Moreover, the multivariate weighted power means $P_{\bw,r}$ in Example \ref{E-multi-power} were
extended in \cite{KL} to the means $P_r(\mu)$ of $\mu\in\cP^\infty(\bP)$, the set of
$\mu\in\cP(\bP)$ with bounded support (i.e., supported on
$\{A\in\bP:\delta I\le A\le\delta^{-1}I\}$ for some $\delta\in(0,1)$), as unique solutions to
the equations for $X\in\bP$
\begin{align}
X&=\int_\bP X\#_rA\,d\mu(A)\qquad\qquad\quad\,\mbox{for $r\in(0,1]$},
\label{power-prob1}\\
X&=\biggl[\int_\bP(X\#_{-r}A)^{-1}\,d\mu(A)\biggr]^{-1}\quad\mbox{for $r\in[-1,0)$}.
\label{power-prob2}
\end{align}
The particular cases where $r=1,-1$ are the arithmetic and harmonic means
\begin{align*}
\cA(\mu):=\int_\bP A\,d\mu(A),\qquad\cH(\mu):=\biggl[\int_\bP A^{-1}\,d\mu(A)\biggr]^{-1}.
\end{align*}
We notice that the above definition of $P_r(\mu)$ is applicable more generally for
$\mu\in\cP(\bP)$ with $\int_\bP(\|A\|+\|A^{-1}\|)\,d\mu(A)<\infty$. Note also that one may extend
the definitions of $G(\mu)$ and $P_r(\mu)$ via equations \eqref{geom-prob}--\eqref{power-prob2}
to the case of an infinite-dimensional $\cH$ (see \cite{LP2}).

The usual order $A\le B$ for $A,B\in\bP$ is naturally extended to $\mu,\nu\in\cP(\bP)$ in such a
way that $\mu\le\nu$ if $\mu(\cU)\le\nu(\cU)$ for any upper Borel subset $\cU$ of $\bP$, where
$\cU$ is said to be \emph{upper} if $A\in\cU$ and $A\le B\in\bP$ imply $B\in\cU$ (see
\cite{La,HLL} for details). Then we say that a mean $M$ on a certain subclass of $\cP(\bP)$ is
monotone if $\mu\le\nu$ implies $M(\mu)\le M(\nu)$ for any $\mu,\nu$ in the domain of $M$.
Moreover, one can consider the weak convergence $\mu_k\to\mu$ for $\mu_k,\mu\in\cP(\bP)$ or the
stronger convergence $d_1^W(\mu_k,\mu)\to0$ for $\mu_k,\mu\in\cP^1(\bP)$ (see \cite{St,HLL} for
details).

With use of the notions mentioned above one can apply the fixed point method considered in the
paper to means of $M$, for example $G$ and $P_r$ above, defined on a subclass of $\cP(\bP)$. Let
$\sigma$ be any operator mean in the sense of Kubo and Ando. For $X\in\bP$ and $\mu\in\cP(\bP)$,
let $X\sigma\mu$ be the push-forward of $\mu$ by the map $A\in\bP\mapsto X\sigma A\in\bP$. Then
it is easy to see that if $\mu$ is, for instance, in the class $\cP^1(\bP)$ or $\cP^\infty(\bP)$,
then $X\sigma\mu$ is in the same class. Hence, for a mean $M$ on such a class, we can consider
the equation
\begin{align*}
X=M(X\sigma\mu),\qquad X\in\bP.
\end{align*}
Under certain conditions of $M$ similar to those given in Section 4, we can show that the above
equation has a unique solution. Thus, the deformation $M_\sigma$ of $M$ by $\sigma$ is defined,
which satisfies properties inherited from the original $M$. By the definition of $P_r$ via
\eqref{power-prob1} and \eqref{power-prob2}, we have examples $P_r=\cA_{\#_r}$ and
$P_{-r}=\cH_{\#_r}$ for $0<r\le1$. The details on the extension of the fixed point method to
probability measures will be presented in a forthcoming paper.


\begin{thebibliography}{99}

\bibitem{ALM}
T. Ando, C.-K. Li and R. Mathias, Geometric means,
{\it Linear Algebra Appl.} {\bf 385} (2004), 305--334.

\bibitem{Bh}
R. Bhatia, {\it Matrix Analysis}, Springer-Verlag, New York, 1996.

\bibitem{BH}
R. Bhatia and J. Holbrook, Riemannian geometry and matrix geometric means,
{\it Linear Algebra Appl.} {\bf 413} (2006), 594--618.

\bibitem{Fu}
J. I. Fujii, Interpolationality for symmetric operator means,
{\it Sci. Math. Jpn.} {\bf 75} (2012), 267--274.

\bibitem{FK}
J. I. Fujii and E. Kamei, Uhlmann's interpolational method for operator means,
{\it Math. Japon.} {\bf 34} (1989), 541--547.

\bibitem{Hi}
F. Hiai, Matrix Analysis: Matrix Monotone Functions, Matrix Means, and Majorization,
{\it Interdisciplinary Information Sciences} {\bf 16} (2010), 139--248.

\bibitem{HK}
F. Hiai and H. Kosaki, Means for matrices and comparison of their norms,
{\it Indiana Univ. Math. J.} {\bf 48} (1999), 899--936.

\bibitem{HLL}
F. Hiai, J. Lawson and Y. Lim, The stochastic order of probability measures on ordered
metric spaces, Preprint, arXiv:1709.04187 [math.FA], 2017.

\bibitem{HL1}
F. Hiai and Y. lim, Log-majorization and Lie-Trotter formula for the Cartan barycenter on
probability measure spaces, {\it J. Math. Anal. Appl.} {\bf 453} (2017), 195--211.

\bibitem{HL2}
F. Hiai and Y. Lim, Geometric mean flows and the Cartan barycenter on the Wasserstein
space over positive definite matrices, {\it Linear Algebra Appl.} {\bf 533} (2017) 118--131.

\bibitem{KL}
S. Kim and H. Lee, The power mean and the least squares mean of probability measures on the
space of positive definite matrices, {\it Linear Algebra Appl.} {\bf 465} (2015), 325--346.

\bibitem{KA}
F. Kubo and T. Ando, Means of positive linear operators,
{\it Math. Ann.} {\bf 246} (1980), 205--224.

\bibitem{La}
J. Lawson, Ordered probability spaces, {\it J. Math. Anal. Appl.} {\bf 455} (2017), 167--179.

\bibitem{LL1}
J. Lawson and Y. Lim, Monotonic properties of the least squares mean,
{\it Math. Ann.} {\bf 351} (2011) 267--279.

\bibitem{LL2}
J. Lawson and Y. Lim, Weighted means and Karcher equations of positive operators,
{\it Proc. Natl. Acad. Sci. USA.} {\bf 110} (2013), 15626--15632.

\bibitem{LL3}
J. Lawson and Y. Lim, Karcher means and Karcher equations of positive definite operators,
{\it Trans. Amer. Math. Soc. Series B} {\bf 1} (2014), 1--22.

\bibitem{LL4}
J. Lawson and Y. Lim, Contractive barycentric maps,
{\it J. Operator Theory}, {\bf 77} (2017), 87--107.

\bibitem{LP1}
Y. Lim and M. P\'alfia, Matrix power means and the Karcher mean,
{\it J. Funct. Anal.} {\bf 262} (2012), 1498--1514.

\bibitem{LP2}
Y. Lim and M. P\'alfia, Existence and uniqueness of the $L^1$-Karcher mean, preprint,
arXiv:1703.04292 [math.FA], 2017.

\bibitem{Mo}
M. Moakher, A differential geometric approach to the geometric mean of symmetric
positive-definite matrices, {\it SIAM J. Matrix Anal. Appl.} {\bf 26} (2005), 735--747.

\bibitem{PP}
M. P\'alfia and D. Petz, Weighted multivariable operator means of positive definite operators,
{\it Linear Algebra Appl.} {\bf 463} (2014), 134--153.

\bibitem{St}
K.-T. Sturm, Probability measures on metric spaces of nonpositive curvature, in
{\it Heat Kernels and Analysis on Manifolds, Graphs, and Metric Spaces (Paris, 2002)},
pp. 357--390, Contemp. Math., 338, Amer. Math. Soc., Providence, RI, 2003.

\bibitem{Th}
A. C. Thompson, On certain contraction mappings in a partially ordered vector space,
{\it Proc. Amer. Math. Soc.} {\bf 14} (1963), 438--443.

\bibitem{UYY}
Y. Udagawa, T. Yamazaki and M. Yanagida, Some properties of weighted operator means and
characterizations of interpolational means, {\it Linear Algebra Appl.} {\bf 517} (2017),

\bibitem{Ya}
T. Yamazaki, An elementary proof of arithmetic-geometric mean inequality of the weighted
Riemannian mean of positive definite matrices,
{\it Linear Algebra Appl.} {\bf 438} (2013), 1564--1569.
217--234.

\end{thebibliography}
\end{document}